\documentclass[11 pt]{amsart}
\usepackage[all]{xy}

\usepackage{amssymb,amsmath,color,mathtools}

\newtheorem{theorem}{Theorem}

\newtheorem{lemma}[theorem]{Lemma}

\newtheorem{conj}[theorem]{Conjecture}
\newtheorem{definition}[theorem]{Definition}
\newtheorem{prop}[theorem]{Proposition}

\theoremstyle{remark}
\newtheorem{remark}{Remark}

\numberwithin{equation}{section}

\def\l{{\mathfrak l}}

\newcommand{\G}{{\Gamma}}

 \let\a\alpha  \let\b\beta    
  
  \let\l\lambda   
      
\let\GL\Lambda
\def\C{\mathbb C}

\def\G{\Gamma}

\def\GL{\mathbf{GL}}
\def \GL2 {{\text{GL}_2}}

\def\Tr{{\rm Tr}}
\def\Gal{{\rm Gal}}
\def\Frob{{\rm Fr}}
\def\F{{\mathbb F}}
  
\def\Z{{\mathbb Z}}

\def\Q{{\mathbb Q}}

\def\G{\Gamma}

\newcommand*\HYPERskip{&}
\catcode`,\active
\newcommand*\pFq{
\begingroup
\catcode`\,\active
\def ,{\HYPERskip}%
\doHyper
}
\catcode`\,12
\def\doHyper#1#2#3#4#5{%
\, _{#1}F_{#2}\left[\begin{matrix}#3 \smallskip \\  #4\end{matrix} \; ; \; #5\right]%
\endgroup
}

\newcommand{\hgp}[4]{
_{2}F_{1} \left(
\begin{matrix}
#1 & #2 \\
   & #3 \\
\end{matrix}
\, ; #4
\right)_p
}

\newcommand{\hgthree}[6]{
\,_{3}F_{2} \left(
\begin{matrix}
#1, & #2, & #3 \\
   & #4, & #5 \\
\end{matrix}
\, ; #6
\right)
}

\newcommand{\fp}
{\mathbb{F}_p}

\newcommand{\fq}
{\mathbb{F}_q}

\newcommand{\fqc}
{\mathbb{F}_q^{\times}}

\newcommand{\fqhat}
{\widehat{\mathbb{F}_{q}^{\times}}}

\def\({\left(}
\def\){\right)}
\def\et #1{\eta_{#1}}
\def\bet #1{\overline\eta_{#1}}
\def\CC#1#2{\binom {#1}{#2}}

\begin{document}

\author{Alyson Deines, Jenny G. Fuselier, Ling Long, Holly Swisher, Fang-Ting Tu}

\address{Center for Communications Research, San Diego, CA 92121, USA}
\email{aly.deines@gmail.com }
\address{High Point University, High Point, NC 27268, USA}
\email{jfuselie@highpoint.edu}
\address{Louisiana State University, Baton Rouge, LA 70803, USA}
\email{llong@math.lsu.edu}
\address{Oregon State University, Corvallis, OR 97331, USA}
\email{swisherh@math.oregonstate.edu}
\address{National Center for Theoretical Sciences, Hsinchu, Taiwan 300, R.O.C.}
\email{ft12@math.cts.nthu.edu.tw}

\title[Truncated hypergeometric and Gaussian hypergeometric functions]{Hypergeometric series, truncated hypergeometric series, and Gaussian hypergeometric functions}



\begin{abstract}In this paper, we investigate the relationships among hypergeometric series, truncated hypergeometric series, and Gaussian hypergeometric functions through some families of `hypergeometric' algebraic varieties that are higher dimensional analogues of Legendre curves.
\end{abstract}

\maketitle

\section{Introduction}
\subsection{Motivation}

When a prime $p$ satisfies $p\equiv 1\pmod 6$, the $p$-adic gamma value $-\G_p\left(\frac 13\right)^3$ is a quadratic algebraic number with absolute value $\sqrt{p}$ which can be written as a Jacobi sum. Thus, $\G_p\left(\frac 13\right)^6$  is not a conjugate of $-\G_p\left(\frac 13\right)^3$ in the sense of algebraic numbers \cite{Cohen2}. However, considering truncated hypergeometric series we have when $p\equiv 1\pmod 6$,
\begin{equation}\label{eq:1}
\pFq{3}{2}{\frac13 , \frac13 , \frac13}{, 1,1}{1}_{p-1}:=\sum_{k=0}^{p-1}\binom{-\frac13}{k}^3 \cdot (-1)^k \equiv \G_p \left(\frac 13 \right)^6  \pmod {p^3},
\end{equation}
which was shown by the third author and Ramakrishna in \cite{LR}, while numerically we see that
\begin{equation}\label{eq:2}
\pFq{3}{2}{\frac23,\frac23,\frac23}{,1,1}{1}_{p-1}:=\sum_{k=0}^{p-1}\binom{-\frac23}{k}^3 \cdot (-1)^k \equiv -\G_p\left(\frac 13\right)^3  \pmod {p^3},
\end{equation}
and we will show this holds modulo $p^2$ in this paper.
By Dwork \cite{Dwork-cycles},
$$\lim_{s\rightarrow \infty}\pFq{3}{2}{\frac13,\frac13,\frac13}{,1,1}{1}_{p^s-1} \Big/ \pFq{3}{2}{\frac13,\frac13,\frac13}{,1,1}{1}_{p^{s-1}-1}=\G_p\left(\frac 13\right)^6,$$
while
$$\lim_{s\rightarrow \infty}\pFq{3}{2}{\frac23,\frac23,\frac23}{,1,1}{1}_{p^s-1} \Big/ \pFq{3}{2}{\frac23,\frac23,\frac23}{,1,1}{1}_{p^{s-1}-1}=-\G_p\left(\frac 13\right)^3.$$
When $p\equiv 5\pmod 6$, Dwork in \cite{Dwork-cycles} showed that there is a similar congruence that involves both $\pFq{3}{2}{\frac13,\frac13,\frac13}{,1,1}{1}_{p^s-1}$ and $\pFq{3}{2}{\frac23,\frac23,\frac23}{,1,1}{1}_{p^{s-1}-1}$.  It is tempting to think of the parameters $\frac 13$ and $\frac 23$ as `conjugates of some sort'.  Also, if one considers the finite field analogue of  $\pFq{3}{2}{\frac13,\frac13,\frac13}{,1,1}{1}$ due to Greene, what corresponds to $\frac 13$ is a cubic character,  which is determined up to a conjugate. Putting these together, it appears that  $-\G\left(\frac 13\right)^3$ is some sort of `conjugate' of $\G\left(\frac 13\right)^6$.  One motivation of this paper is to investigate these seemingly contradicting phenomena via the relations between hypergeometric series, Gaussian hypergeometric functions and truncated hypergeometric series. These objects correspond to periods, Galois representations, and unit roots (in the ordinary case) respectively.

In recent work \cite{DFLST}, the authors use the perspective of  Wolfart \cite{Wolfart} and Archinard \cite{Archinard} to consider classical $_2F_1$-hypergeometric functions with rational parameters as periods of explicit generalized Legendre curves
$$
C_{\lambda}^{[N;i,j,k]}: y^N=x^i(1-x)^j(1-\l x)^k.
$$
In \cite{DFLST}, the main players are hypergeometric series and Gaussian hypergeometric functions.  The authors use Gaussian $_2F_1$-hypergeometric  functions to count points of  $C_{\lambda}^{[N;i,j,k]}$ over finite fields and hence compute the corresponding Galois representations. This arithmetic information together with the periods of  $C_{\lambda}^{[N;i,j,k]}$ in terms of hypergeometric values yields information about the decomposition of the Jacobian variety $J_{\lambda}^{[N;i,j,k]}$ constructed from the desingularization of $C_{\lambda}^{[N;i,j,k]}$. When $\gcd(i,j,k)$ is coprime to $N$ and $N\nmid i+j+k$,  then $J_{\lambda}^{[N;i,j,k]}$ has a degree $2\varphi(N)$ `primitive' factor $J_\l^{new}$, where $\varphi$ is the Euler phi function.  The authors prove the following theorem.

\begin{theorem}[\cite{DFLST}]\label{thm:1}
Let $N=3,4,6$ and $1\le i,j,k<N$ with $\gcd(i,j,k)$ coprime to $N$ and $N\nmid i+j+k$.   Then for each $\l \in \overline{\Q}$, the endomorphism algebra of $J_\l^{new}$  contains a 4-dimensional algebra over $\Q$  if and only if
\[
B \left (\frac{N-i}N,\frac{N-j}N \right )\Big / B \left (\frac kN, \frac{2N-i-j-k}N \right )\in \overline \Q,
\] where $B(a,b)=\frac{\G(a)\G(b)}{\G(a+b)}$, and $\G(\cdot)$ is the Gamma function.
\end{theorem}

The second motivation for this paper is to explore the following higher dimensional analogues of Legendre curves
$$C_{n,\l}: \quad y^n=(x_1x_2\cdots x_{n-1})^{n-1}(1-x_1)\cdots(1-x_{n-1})(x_1-\l x_2x_3\cdots x_{n-1}).$$ In particular, the curves $C_{2,\l}$ are known as Legendre curves. Up to a scalar multiple, the hypergeometric series
\[
\pFq{n}{n-1}{\frac jn, \frac jn,\cdots, \frac jn }{,1,\cdots,1}{\l}
\]
for any $1\le j\le n-1$, when convergent, can be realized as a period of $C_{n,\l}$.

\subsection{Results}Our first theorem shows that the number of rational points on $C_{n,\l}$ over finite fields $\fq$ can be expressed in terms of Gaussian hypergeometric functions.  For a definition of Gaussian hypergeometric functions please see Section \ref{GaussSec}.  \footnote{The subscript $q$ for a Gaussian hypergeometric function  records the size of the corresponding finite field and should not be confused with the subscript for truncated hypergeoemetric series which records the location of truncation.} Let $\widehat{\mathbb{F}_{q}^{\times}}$ denote the group of all multiplicative characters on $\fqc$.

\begin{theorem}\label{thm:2}
Let  $q=p^e\equiv 1\pmod{n}$ be a prime power. Let $\eta_n$ be a primitive order $n$ character   and $\varepsilon$ the trivial multiplicative character in $\fqhat$. Then

  $$
      \#C_{n,\l}(\fq)=1+q^{n-1}+ q^{n-1} \, \sum_{i=1}^{n-1}  \, _nF_{n-1} \(\begin{array}{cccc} {\eta_n^i,}&{\eta_n^i,}&\cdots,&{\eta_n^i,}\\& {\varepsilon,}& \cdots,&{\varepsilon,}\end{array}{;\l} \)_q.
     $$
\end{theorem}

Meanwhile, we are also interested in knowing how to use information from truncated hypergeometric series to obtain information about the Galois representations and hence local zeta functions of $C_{n,\l}$.  For instance, we have the following conjecture based on numerical evidence for the case $\l=1$.

\begin{conj}\label{thm:4}
Let $n\geq 3$ be a positive integer, and $p$ be prime such that $p\equiv 1 \pmod{n}$. Then
\[
\pFq{n}{n-1}{\frac{n-1}{n} & \frac{n-1}{n}  & \ldots & \frac{n-1}{n}}{ & 1 & \ldots  & 1}{1}_{p-1}:=\sum_{k=0}^{p-1} \binom{\frac{1-n}n}{k}^n(-1)^{kn}\equiv -\G_p\left(\frac{1}{n}\right)^n \pmod{p^3}.
\]
\end{conj}

\noindent Using the Gross-Koblitz formula \cite{GK}, recalled in Section \ref{GKSec}, we have for a prime $p\equiv 1\pmod n$,
$$J(\eta_n,\eta_n)J(\eta_n,\eta_n^2)\cdots J(\eta_n,\eta_n^{n-2})=(-1)^{n-2+ \frac{1+(n-1)p}n}\G_p\left(\frac 1n\right)^n,$$
where $J(\cdot,\cdot)$ denotes the standard Jacobi sum.   We see that $(-1)^{n-2+ \frac{1+(n-1)p}n}=1$ when $n$ is odd and $\eta_n$ is an order $n$ character of $\F_p^\times$ such that $\eta_n(x)\equiv x^{\frac{p-1}n} \pmod p$ for all $x\in \F_p$. From the perspective of Gr\"ossencharacters (Hecke characters) (see Weil \cite{Weil}), this product of Jacobi sums is associated with a linear representation $\chi$ of the Galois group $\Gal(\overline \Q/\Q(e^{2\pi i/n}))$.  We would like to explore whether the Galois representation arising from $C_{n,1}$ contains a factor that is related to $\chi$. When $n=3,4$, the answer is positive. In proving these results the work of Greene \cite{Greene} and  McCarthy  \cite{McCarthy} on finite field analogues of classical hypergeometric evaluation formulas plays an essential role.

Ahlgren and Ono \cite{Ahlgren-Ono} show that for any odd prime $p$,
\begin{equation}\label{AOresult}
p^3\cdot\, _4F_3\(\begin{matrix}
{\eta_4^2}, & {\eta_4^2},& {\eta_4^2},&{\eta_4^2}  \\
   & \varepsilon,  & \varepsilon, & \varepsilon\\
\end{matrix};1\)_p=-a(p)-p,
\end{equation}
where $a(p)$ is the $p$th coefficient of the weight 4 Hecke cuspidal eigenform $$\eta(2z)^4\eta(4z)^4,$$ with $\eta(z)$ being the Dedekind eta function.
Here, we show the following.
\begin{theorem}\label{GaussianTheorem}
Let $\eta_2$, $\eta_3$ or $\eta_4$  denote characters of order 2, 3, or 4, respectively, in $\fqhat$.

\begin{enumerate}
\item \label{Gpart1}
Let $q\equiv 1 \pmod 3$ be a prime power.  Then
$$
q^2\cdot\,  \hgthree{\eta_3}{\eta_3}{\eta_3}{\varepsilon}{\varepsilon}{1}_q = J(\eta_3,\eta_3)^2-J(\eta_3^2,\eta_3^2).
$$
\item \label{Gpart2}
Let $q\equiv 1 \pmod 4$ be a prime power. Then
$$
q^3\cdot\, _4F_3\(\begin{matrix}
{\eta_4}, & {\eta_4},& {\eta_4},&{\eta_4}  \\
   & \varepsilon,  & \varepsilon, & \varepsilon\\
\end{matrix};1\)_q
=J({\eta_4},\eta_2)^3+qJ({\eta_4},\eta_2)-J(\overline{\eta_4},\eta_2)^2.
$$
\end{enumerate}
\end{theorem}Here we observe $J(\overline{\eta_4},\eta_2)^2=\eta_4(-1)J(\overline{\eta_4},\overline{\eta_4})J(\overline{\eta_4},\overline{\eta_4}^2)$.
To prove Theorem \ref{GaussianTheorem} we use the work of Greene \cite{Greene} and McCarthy \cite{McCarthy}, except in case (\ref{Gpart2}) when $q\equiv 5 \pmod 8$, in which we use Gr\"ossencharacters and representation theory.   The reason we do this is because a key ingredient of our proof is Theorem 1.6 of McCarthy \cite{McCarthy}, for which we assume $\eta_4$ is a square, i.e., $q\equiv 1\pmod 8$. Combining this with the theory of Galois representations, we can reach our conclusion when $q\equiv 5\pmod 8$.   We wish to point out that the above results can be interpreted in terms of Galois representations.
\footnote{ In a different language, our results correspond to the explicit descriptions of some mixed weight hypergeometric motives arising from exponential sums which are initiated by Katz \cite{Katz}, and are explicitly formulated and implemented by a group of mathematicians including Beukers, Cohen, Rodriguez-Villegas and others (from private communication with H. Cohen and F. Rodriguez-Villegas).  Here we can use the explicit algebraic varieties to compute the Galois representations directly. A different algebraic model for the algebraic varieties is given in the following recent preprint \cite{BCM}. 
}
Result (1) describes the trace of the Frobenius element at $q$ in $\Gal(\overline \Q/\Q(\sqrt{-3}))$ under a 2-dimensional Galois representation arising from the second \'etale cohomology of $C_{3,1}$ in terms of Jacobi sums (and hence Gr\"ossencharacters); while (2) describes a 3-dimensional Galois representation of $\Gal(\overline \Q/\Q(\sqrt{-1}))$ arising from the third \'etale cohomology of $C_{4,1}$ in terms of Jacobi sums. Both cases are exceptional.  Consequently we can describe the local zeta function of $C_{3,1}$ and $C_{4,1}$ completely. For instance when $p\equiv 1 \pmod 3$ is prime, by the Hasse-Davenport relation for Jacobi sums (see \cite{IR90}), the local zeta function of $C_{3,1}$ over $\F_p$ is
$$ Z_{C_{3,1}}(T,p)= \frac{1}{(1-T)(1+(\alpha_p+\overline{\alpha}_p)T+pT^2)(1-p^2T)(1-(\alpha_p^2+\overline{\alpha}_p^2)T+p^2T^2)}$$
where  $\alpha_p=J(\eta_3,\eta_3)$. Note that the factor $(1+(\alpha_p+\overline{\alpha}_p)T+pT^2)$ appearing in the denominator has roots of absolute value $1/\sqrt{p}$; meanwhile following Weil's conjecture (see \cite{IR90}) such a term should appear  in the numerator instead. We believe the discrepancy is due to the fact that we are not computing using the smooth model of $C_{n,\l}$ as no resolution of singularities is involved so far. Similarly, we have for any prime $p\equiv 1\pmod 4$

\begin{multline*}Z_{C_{4,1}}(T,p)= 
(1+(\b_p^3+\overline{\b}_p^3)T+p^3T^2)(1+(\b_p+\overline{\b})pT+p^3T^2) \\
\cdot\frac{(1-(\b_p^2+\overline{\b}_p^2)T+p^2T^2)(1-a(p)T+p^3T^2)(1-pT) }{(1-T)(1-p^3T)},
\end{multline*} 
where $a(p)$ as in \eqref{AOresult} and $\beta_p=J(\eta_4,\eta_2)$.
The factor corresponding to $$y^2=(x_1x_2x_3)^3(1-x_1)(1-x_2)(1-x_3)(x_1-x_2x_3)$$ is
$$
Z_{C^{old}_{4,1}}(T,p)= \frac{(1-a(p)T+p^3T^2)(1-pT) }{(1-T)(1-p^3T)},
$$
and new primitive portion is
\begin{multline*}
Z_{C^{new}_{4,1}}(T,p)= 
{(1+(\b_p^3+\overline{\b}_p^3)T+p^3T^2)(1+(\b_p+\overline{\b})pT+p^3T^2)(1-(\b_p^2+\overline{\b}_p^2)T+p^2T^2). }
\end{multline*}

Part (1) of Theorem \ref{GaussianTheorem} explains why $-\G_p(\frac 13)^3$ appears to be a conjugate of $\G_p(\frac 13)^6$. There are two ways to specify a cubic character in $\widehat {\F_p^\times}$ when $p\equiv 1\pmod 3$, i.e. $\eta_3(x)\equiv x^{(p-1)/3}\pmod p$ for all $x\in \F_p$ or $\eta_3(x)\equiv x^{2(p-1)/3}\pmod p$.  Either way gives an embedding of
$$p^2\cdot  \hgthree{\eta_3}{\eta_3}{\eta_3}{\varepsilon}{\varepsilon}{1}_p$$
to $\Z_p$. Then the image of the Gaussian hypergeometric function is congruent to $-\G_p(\frac 13)^3$ or $\G_p(\frac 13)^6$ respectively via the Gross-Koblitz formula \cite{GK, GK-revisited}. Using this formula, we also prove the following result which relates Gaussian hypergeometric functions to truncated hypergeometric series.

\begin{lemma}\label{GaussToHyper} Let $r,n,j$ be positive integers with $1\le j<n$. Let $p\equiv 1 \pmod n$ be prime and $\eta_n\in \widehat{\F_p^\times}$ such that $\eta_n(x)\equiv x^{j(p-1)/n} \pmod p$ for each $x\in \F_p$.  Then,
\begin{multline*}
p^{r-1}\cdot{} _{r}F_{r-1}\left(
\begin{matrix}
\eta_n, & \eta_n, & \cdots, & \eta_n \\
   & \varepsilon, & \cdots, & \varepsilon \\
\end{matrix}
\, ; x
\right)_p
\equiv \\
(-1)^{r+1}\cdot \pFq{r}{r-1}{\frac{n-j}n& \frac{n-j}n& \cdots& \frac{n-j}n}{& 1& \cdots &1}{\frac 1x}_{(p-1)\left(\frac{n-j}n\right)}
\\+(-1)^{r+1+\frac{(p-1)}njr}\(x^{(p-1)\frac{n-j}n}-x^{\frac{p-1}nj}\)  \pmod p;
\end{multline*}

\begin{multline*}
p^{r-1}\cdot{} _{r}F_{r-1}\left(
\begin{matrix}
\overline{\eta_n}, & \overline{\eta_n}, & \cdots, & \overline{\eta_n} \\
   & \varepsilon, & \cdots, & \varepsilon \\
\end{matrix}
\, ; x
\right)_p
\equiv \\
(-1)^{r+1}\cdot p^r   \pFq{r+1}{r}{1 & 1 & \cdots & 1}{& \frac{2n-j}{n} & \cdots & \frac{2n-j}{n} }{\frac 1x}_{p-1} \pmod{p}.
\end{multline*}
\end{lemma}

Thus,  \eqref{eq:1} and  \eqref{eq:2} hold modulo $p$.  It is shown in \cite{LR} that \eqref{eq:1} holds modulo $p^3$.  These kinds of stronger congruences are known as \emph{supercongruences} as they are stronger than what the theory of formal groups can predict. We will establish a few here.  In particular, we prove the claim that Conjecture \ref{thm:4} holds modulo $p^2$.

\begin{theorem}\label{thm:4a} Conjecture \ref{thm:4} holds modulo $p^2$. Namely, for $n\geq 3$, and $p\equiv 1 \pmod{n}$ prime,
\[
\pFq{n}{n-1}{\frac{n-1}{n} & \frac{n-1}{n}  & \ldots & \frac{n-1}{n}}{ & 1 & \ldots  & 1}{1}_{p-1}:=\sum_{k=0}^{p-1} \binom{\frac{1-n}n}{k}^n(-1)^{kn}\equiv -\G_p\left(\frac{1}{n}\right)^n \pmod{p^2}.
\]

\end{theorem}

\noindent \emph{Remark.}   Theorem \ref{thm:4a} also holds for $n=2$, due to Mortenson \cite{Mortenson}.

We note that in \cite[Defn. 1.4]{McCarthy3}, McCarthy defines a new function $_nG_n[\cdots]$ in terms of sums and ratios of $p$-adic Gama functions. Recently, the second author and McCarthy produced families of congruences between these $_nG_n$ functions and truncated hypergeometric series \cite{Fuselier-McCarthy}. New identities for these functions have also recently been obtained  by McCarthy, et. al. \cite{McCarthy-Barman-S}  and it is possible they could be used to prove Conjecture \ref{thm:4} in full.

\medskip

For the truncated hypergeometric series related to $C_{4,1}$ we have another result.

\begin{theorem}\label{H4}
For each prime $p\equiv 1 \pmod 4$,
\[
\pFq{4}{3}{\frac14 & \frac14 & \frac14 & \frac14}{&1&1&1}{1}_{p-1}=\sum_{k=0}^{p-1}\binom{-\frac14}{k}^4 \equiv
 (-1)^{\frac{p-1}4}\G_p\left(\frac12\right) \G_p\left(\frac 14\right)^6 \pmod {p^4}.
\]
\end{theorem}

Corresponding to Ahlgren and Ono's result \eqref{AOresult}, Kilbourn \cite{Kilbourn} shows the supercongruence
$$\pFq{4}{3}{\frac 12,\frac 12,\frac 12,\frac12}{,1,1,1}{1}_{p-1}:=\sum_{k=0}^{p-1} \binom{-\frac 12}{k}^4\equiv a(p)\pmod {p^3},$$
where $a(p)$ is defined as in \eqref{AOresult}.

Supercongruences are not only intellectually appealing, they are of very practical use for our computations. For instance, Theorem \ref{H4} corresponds to the properties of the third \'etale cohomology group of  $C_{4,1}$ as mentioned earlier. By the Hasse-Weil bounds for them, which are constant multiplies of $p$ and $p^{3/2}$ respectively, the supercongruence results allowed us to compute the traces of Frobenius without any ambiguity, from which we were able to nail down the local zeta functions of $C_{3,1}$ and $C_{4,1}$ and discover Theorem \ref{GaussianTheorem} numerically before proving it. There are a variety of different techniques for proving such results and each has its own strength. See \cite{CDLNS} for another Women in Numbers (WIN) project on supercongruences, which was motivated by the work of Zudilin \cite{Zudilin} and his conjectures.
 We prove Theorems \ref{thm:4a} and \ref{H4} by deforming truncated hypergeometric series using hypergeometric evaluation identities (several of them are due to Whipple \cite{Whipple, AAR}) together with $p$-adic analysis via harmonic sums and $p$-adic Gamma functions. This technique is originated in \cite{CLZ} and is later formulated explicitly in \cite{LR}.

\subsection{Outline of this paper}
Section \ref{Prelim} contains some background material.  In Section \ref{Leg}, we consider the familiar setting of Legendre curves. This section serves as a showcase of our techniques without getting into too much technicality.  We prove Theorem \ref{thm:2} and Lemma \ref{GaussToHyper} in Section \ref{ss:proofofthm4}. Section \ref{Gauss} is devoted to proving the results on Gaussian hypergeometric functions in Theorem \ref{GaussianTheorem}.  In Section \ref{Super}, we prove Theorem \ref{thm:4a}, based on an idea of Zudilin (private communication), and then prove Theorem \ref{H4}.   Sections \ref{ss:proofofthm4}, \ref{Gauss}, and \ref{Super} are technical in nature. In Section \ref{Rem} we end with some remarks including a few conjectures based on our numerical data computed using \texttt{Sage}.

\section{Preliminaries}\label{Prelim}
\subsection{Generalized hypergeometric series and truncation}
For a positive integer $r$, and $\alpha_i$, $\beta_i \in \mathbb{C}$ with $\beta_i\not\in\{\ldots, -3, -2, -1\}$, the (generalized) hypergeometric series $_{r}F_{r-1}$ is defined by
\[
\pFq{r}{r-1}{\alpha_1 & \alpha_2 & \ldots & \alpha_{r}}{& \beta_1 & \ldots  & \beta_{r-1}}{\l} :=
\sum_{k= 0}^{\infty} \frac{(\alpha_1)_k(\alpha_2)_k \ldots (\alpha_{r})_k}{(\beta_1)_k
\ldots (\beta_{r-1})_k} \cdot \frac{\l^k}{k!}
\]
where $(a)_{0}:=1$ and $\displaystyle{(a)_k :=a(a+1)\cdots(a+k-1)}$  are rising factorials.   This series converges for $|\l|<1$.

When we truncate the above sum at $k=m$, we use the subscript notation
\[
\pFq{r}{r-1}{\alpha_1 & \alpha_2 & \ldots & \alpha_{r}}{ & \beta_1 & \ldots  & \beta_{r-1}}{\l}_{m} :=
\sum_{k= 0}^{m} \frac{(\alpha_1)_k(\alpha_2)_k \ldots (\alpha_{r})_k}{(\beta_1)_k
\ldots (\beta_{r-1})_k} \cdot \frac{\l^k}{k!}.
\]
We note that the books by Slater \cite{Slater}, Bailey \cite{Bailey}, and Andrews, Askey and Roy \cite{AAR} are excellent sources for information on classical hypergeometric series.

\noindent The following gives an alternate truncation for hypergeometric series modulo powers of primes.
\begin{lemma}\label{AltTrunc}
Let $n\geq 2$ be a positive integer, $j$ an integer $1\le j<n$, and $p\equiv 1 \pmod{n}$ prime.  Then for $x\in \Z_p$,
\[
\pFq{r}{r-1}{\frac{j}{n} & \frac{j}{n}  & \cdots & \frac{j}{n}}{ & 1 & \cdots  & 1}{x}_{\frac{j}{n}(p-1)} \equiv \pFq{r}{r-1}{\frac{j}{n} & \frac{j}{n}  & \cdots & \frac{j}{n}}{ & 1 & \cdots  & 1}{x}_{p-1} \pmod{p^r}.
\]

\end{lemma}

\begin{proof}
The lemma follows from the fact that when $\frac{j(p-1)}n+1\leq k \leq (p-1)$, the rising factorial $\(\frac jn \right)_k\in p\Z_p$, since it contains the factor $p\left( \frac jn\right)$, while $(1)_k$ is not divisible by $p$.
\end{proof}

\subsection{Euler's integral formula and higher generalization}
When $\rm{Re}(\beta_1)>\rm{Re}(\alpha_2)>0$, Euler's integral representation for $_2F_1$ \cite{AAR} states that
\[
\pFq{2}{1}{\alpha_1,\alpha_2}{,\beta_1}{\l} = \frac{\G(\beta_1)}{\G(\alpha_2)\G(\beta_1-\alpha_2)} \int_0^1 x^{\alpha_2-1}(1-x)^{\beta_1-\alpha_2-1}(1-\l x)^{-\alpha_1}dx.
\]
More generally, one has that when $\text{Re} \left(\beta_r \right)>\text{Re} \left(\a_{r+1} \right)>0$ (see \cite[(2.2.2)]{AAR})
\begin{multline}\label{eq:gen-Euler-int}
\pFq{r+1}{r}{\alpha_1 & \alpha_2 & \ldots & \alpha_{r+1}}{ & \beta_1 & \ldots  & \beta_{r}}{\l} =  \frac{\G(\beta_r)}{\G(\alpha_{r+1})\G(\beta_r-\alpha_{r+1})} \\
\cdot  \int_0^1 x^{\alpha_{r+1}-1}(1-x)^{\beta_r -\alpha_{r+1} -1} \pFq{r}{r-1}{\alpha_1 & \alpha_2 & \ldots & \alpha_{r}}{ & \beta_1 & \ldots  & \beta_{r-1}}{\l x}dx.
\end{multline}
From the above two integral formulas, one can derive that for each $1\le j\le n-1$ the series $\pFq{n}{n-1}{\frac jn&\frac jn&\cdots&\frac jn}{&1&\cdots &1}{\l}$, with a suitable beta quotient factor, is a period of $C_{n,\l}$.

\subsection{Gaussian hypergeometric functions}\label{GaussSec}

Let $p$ be prime and let $q=p^e$.
We extend any character $\chi\in\ \widehat{\mathbb{F}_{q}^{\times}}$ to all of $\fq$ by setting $\chi(0)=0$, including the trivial character $\varepsilon$, so that $\varepsilon(0)=0$.  For $A,B\in\fqhat$, let $J(A,B):=\sum_{x\in \fq}A(x)B(1-x)$ denote the Jacobi sum and define

\begin{equation*}
\binom{A}{B}:=\frac{B(-1)}{q}J(A,\overline{B})=\frac{B(-1)}{q} \sum_{x\in \fq} A(x)\overline{B}(1-x).
\end{equation*}

In \cite{Greene}, Greene defines a finite field analogue of hypergeometric series called Gaussian hypergeometric functions, defined below.
\begin{definition}[\cite{Greene} Defn. 3.10] \label{def: HGF-finite}
If $n$ is a positive integer, $x\in\fq$, and $A_0,A_1,\dots,A_n,$\\$B_1,B_2,\dots,B_n \in \widehat{\mathbb{F}_{q}^{\times}}$, then
$$ _{n+1}F_{n} \left(
\begin{matrix}
A_0, & A_1, & \dots, & A_n \\
     & B_1, & \dots, & B_n \\
\end{matrix}
; x
\right)_q
:= \frac{q}{q-1} \sum_{\chi\in\widehat{\mathbb{F}_{q}^{\times}}} \binom{A_0\chi}{\chi} \binom{A_1\chi}{B_1\chi} \dots \binom{A_n\chi}{B_n\chi} \chi(x).$$
\end{definition}


Greene showcases a variety of identities satisfied by his Gaussian hypergeometric functions, many of which provide direct analogues for transformations of classical hypergeometric series. For example, he provides a finite field analogue of \eqref{eq:gen-Euler-int}, shown below.

\begin{theorem}[Greene \cite{Greene}]\label{Greene-3.13}
For characters $A_0,A_1,\dots,A_n,B_1,\dots,B_n$ in $\fqhat$, and $x\in\fq$,
\begin{multline*}
_{n+1}F_{n} \left(
\begin{matrix}
A_0, & A_1, & \dots, & A_n \\
     & B_1, & \dots, & B_n \\
\end{matrix}
; x
\right)_q
= \\
\frac{A_nB_n(-1)}{q}  \cdot
\sum_{y}  {}_nF_{n-1} \left( \begin{matrix}
A_0, & A_1, & \dots, & A_{n-1} \\
     & B_1, & \dots, & B_{n-1} \\
\end{matrix}
; xy
\right)_q
\cdot A_n(y)\overline{A_n}B_n(1-y).
\end{multline*}
\end{theorem}

To extend Greene's program, McCarthy provides a modification of Greene's functions below.  We let $g(\chi)=\displaystyle\sum_{x\in\F_q}\chi(x)\zeta_p^{\mathrm{Tr}(x)}$ denote the Gauss sum of $\chi$, and $\Tr$ the usual trace map form $\F_q$ to $\F_p$.

\begin{definition}\cite{McCarthy} For characters $A_0,A_1,\dots,A_n,B_1,\dots,B_n$ in $\fqhat$,
\begin{multline*}
  {}_{n+1}F_n\(\begin{matrix}
A_0, &A_1,&\ldots,&A_n  \\
   & B_1,  & \ldots,&B_n\\
\end{matrix};x\)_q^\ast
:= 
\frac1{q-1}\sum_{\chi\in\fqhat}\prod_{i=0}^n\frac{g(A_i\chi)}{g(A_i)}\prod_{j=1}^n\frac{g(\overline{B_j\chi})}{g(\overline{B_j})}g(\overline\chi)
\chi(-1)^{n+1}\chi(x).
\end{multline*}
\end{definition}

\noindent McCarthy makes explicit how the two hypergeometric functions are related, via the following.

\begin{prop}[McCarthy \cite{McCarthy}]\label{Greene-McCarthy} If $A_0\neq \varepsilon$ and $A_i\neq B_i$ for each $1\leq i \leq n$, then
$$   {}_{n+1}F_n\(\begin{matrix}
A_0, &A_1,&\ldots,&A_n  \\
   & B_1,  & \ldots,&B_n\\
\end{matrix};x\)_q^\ast
= \left[ \prod_{i=1}^n \binom{A_i}{B_i}^{-1}\right]   {}_{n+1}F_n\(\begin{matrix}
A_0, &A_1,&\ldots,&A_n  \\
   & B_1,  & \ldots,&B_n\\
\end{matrix};x\)_q. $$

\end{prop}

McCarthy uses this hypergeometric function to provide analogues to classical formulas of Dixon, Kummer, and Whipple for well-poised classical hypergeometric series \cite{McCarthy}. For example, consider Whipple's classical transformation below:

\begin{theorem}[Whipple \cite{Whipple}] If one of $1+\frac12 a-b$, $c$, $d$, $e$ is a negative integer, then
\begin{align*}
& \hspace*{-0.5in}_5F_{4} \left[ \begin{matrix}
a & b& c & d & e \\
     & 1+a-b&1+a-c&1+a-d & 1+a-e\,\, \\
\end{matrix}
; 1
\right]    \\
&= \frac{\Gamma(1+a-c)\Gamma(1+a-d)\Gamma(1+a-e)\Gamma(1+a-c-d-e)}{\Gamma(1+a)\Gamma(1+a-d-e)\Gamma(1+a-c-d)\Gamma(1+a-c-e)}\,   \\
& \hspace*{1.0in} \cdot \,_4F_{3}\left[ \begin{matrix}
1+\frac12 a-b & c& d &e   \\
     & 1+\frac12 a & c+d+e-a & 1+a-b\,\, \\
\end{matrix}
; 1
\right]. \end{align*}
\end{theorem}

\noindent McCarthy provides a finite field analogue to this result using his hypergeometric series.

\begin{theorem}[McCarthy, Thm. 1.6 of \cite{McCarthy}]\label{McCarthy5F4} Let $A,B,C,D,E\in\fqhat$ such that, when $A$ is a square, $A\neq\varepsilon$, $B\neq\varepsilon$, $B^2\neq A$, $CD\neq A$, $CE\neq A$, $DE\neq A$, and $CDE\neq A$.  Then, if $A$ is not a square,
\[
{}_{5}F_4\(\begin{matrix}
A, &B,& C,& D,& E  \\
 & A\overline{B }, & A\overline{C},& A\overline{D} & A\overline{E} \\
\end{matrix};1\)_q^\ast 
= 0,
\]
and if $A$ is a square,
\begin{multline*}
 {}_{5}F_4\(\begin{matrix}
A, &B,& C,& D,& E  \\
 & A\overline{B }, & A\overline{C},& A\overline{D} & A\overline{E} \\
\end{matrix};1\)_q^\ast = \\
\displaystyle \frac{g(\overline{A})g(\overline{A}DE)g(\overline{A}CD)g(\overline{A}CE)}{g(\overline{A}C)g(\overline{A}D)g(\overline{A}E)g(\overline{A}CDE)} \displaystyle \sum_{R^2=A} {}_{4}F_3\(\begin{matrix}
R\overline{B}, &C,& D, &E  \\
 & R & \overline{A}CDE,& A\overline{B} \\
\end{matrix};1\)_q^\ast  \\
+\displaystyle\frac{g(\overline{A}DE)g(\overline{A}CD)g(\overline{A}CE)q}{g(C)g(D)g(E)g(\overline{A}C)g(\overline{A}D)g(\overline{A}E)} \, {}_{2}F_1\(\begin{matrix}A, &B \\ & A\overline{B} \\ \end{matrix};-1\)_q^\ast .
\end{multline*}
\end{theorem}

Gaussian hypergeometric functions have been used to count points on different types of varieties over $\fq$ and they are related to coefficients of various modular forms \cite{Koike,Ono,FOP,Fuselier,Lennon-count,Ve11,Barman, Ahlgren-Ono}.  We use Greene's hypergeometric functions to count points on $C_{n,\lambda}$ in Section 4.1. We make use of McCarthy's hypergeometric function, as well as the previous theorem, Theorem \ref{McCarthy5F4}, in the proof of Theorem \ref{GaussianTheorem} in Section \ref{Gauss}. Values of McCarthy's normalized version of the hypergeometric function over finite fields have also been shown to be related to eigenvalues of Siegel modular forms of higher degree \cite{MP}.

\subsection{$p$-adic Gamma functions and the Gross-Koblitz formula}\label{GKSec}

We first recall the $p$-adic $\Gamma$-function. The $p$-adic $\Gamma$-function is defined for $n\in\mathbb{N}$ by
$$ \Gamma_p(n) := (-1)^{n}\prod_{\substack{0<  j < n \\ p \, \nmid j}}j,$$
and extends to $x\in\mathbb{Z}_p$ by defining $\Gamma_p(0):=1$, and for $x\neq 0$,
$$\Gamma_p(x):=\lim_{n\rightarrow x}\Gamma_p(n),$$
where $n$ runs through any sequence of positive integers $p$-adically approaching $x$.

We recall some basic properties for $\G_p(\cdot)$ which will be useful later (see Theorem 14 of \cite{LR}). 
\begin{prop}\label{GpFacts}
Let $x\in \Z_p$.  We have the following facts about $\G_p$.
\begin{itemize}
\item[a)] $\G_p(0)=1$
\item[b)] $\G_p(x+1)/\G_p(x)=-x$ unless $x\in p\Z_p$ in which case the quotient takes value $-1$.
\item[c)] $\G_p(x)\G_p(1-x)=(-1)^{a_0(x)}$ where $a_0(x)$ is the least positive residue of $x$ modulo $p$,
\item[d)] Given $p>11$, there exist $G_1(x), G_2(x)\in \Z_p$ such that for any $m\in \Z_p$,
\begin{equation*}
\G_p(x+mp)\equiv \G_p(x)\left[1+G_1(x)mp+G_2(x)\frac{(mp)^2}2+G_3(x)\frac{(mp)^3}6\right] \pmod{p^4}.
\end{equation*}
\item[e)] $G_1(x)=G_1(1-x)$ and $G_2(x)+G_2(1-x)=2G_1(x)^2$,
\item[f)] If $x \equiv y \pmod{p^n}$, then $\Gamma_p(x) \equiv \Gamma_p(y) \pmod{p^n}$.
\end{itemize}
\end{prop}

\noindent We note that c) above implies in particular that for any integer $n>1$ and prime $p\equiv 1\pmod n$, $$\G_p\left(\frac 1n\right)\G_p\left(1-\frac 1n\right)=(-1)^{\frac{1+(n-1)p}n}.$$ Thus when $n$ is odd, $\G_p(\frac 1n)\G_p(1-\frac 1n)=-1$.

We now recall the Gross-Koblitz formula \cite{GK, GK-revisited} in the case of $\F_p$.  Let $$\varphi: \F_p^\times \rightarrow \Z_p^\times$$ be the Teichm\"uller character such that $\varphi(x)\equiv x \pmod p$. The Gross-Koblitz formula states that the Gauss sum $g(\varphi^{-j})$ defined using the Dwork exponential as the additive character satisfies
\begin{equation}\label{eq:G-K}
g\(\varphi^{-j}\)=-\pi_p^j\G_p\(\frac{j}{p-1}\),
\end{equation}
where $0\le j\le p-2,$ and $\pi_p \in \C_p$ is a root of $x^{p-1}+p=0$.

\section{In the setting of Legendre curves}\label{Leg}
We first briefly explain the relationships between Gaussian hypergeometric functions, hypergeometric series, and truncated hypergeometric series using the Legendre curve
$$C_{2,\l}: y^2=x_1(1-x_1)(x_1-\l),$$
which is isomorphic to the more familiar form $y^2=x(x-1)(x-\l)$ over $\Q(\sqrt{-1})$. It is well-known that one period of $C_{2,\l}$ is given by
$$
\pi \cdot \pFq{2}{1}{\frac 12,\frac 12}{,1}{\l}.
$$

Assume $\l\in \Q$ and $\eta_2\in \widehat{\F_p^\times}$ is  of order $2$. It follows from the Taniyama-Shimura-Wiles theorem \cite{Wiles} that for good primes $p$,
$$a_p(\l)=p+1-\#C_{2,\lambda}(\fp)= -\!\!\!\! \sum_{x_1\in \F_p} \eta_2(x_1(1-x_1)(x_1-\l )) = -p \cdot \hgp{\eta_2,}{\eta_2}{\varepsilon}{\l}$$
is the $p$th coefficient of a weight 2 cuspidal Hecke eigenform that can be computed from a compatible family of 2-dimensional $\ell$-adic Galois representations of $G_\Q:=\Gal(\overline \Q/\Q)$ constructed from the Tate module of $C_{2,\l}$ via L-series. This gives a correspondence between the $_2F_1$ Gaussian hypergeometric functions and the Galois representations arising from $C_{2,\l}$.

When
$a_p(\l)$ is not divisible by $p$, i.e. $p$ is ordinary for $C_{2,\l}$, then a result of Dwork \cite{Dwork-cycles} says that
\begin{equation}\label{eq: 2F1-1/2}
 \displaystyle \lim_{s\rightarrow \infty}\pFq{2}{1}{\frac 12&\frac 12}{&1}{\hat \l}_{p^s-1} \Big/ \pFq{2}{1}{\frac 12&\frac 12}{&1}{\hat \l}_{p^{s-1}-1}
\end{equation}
is the unit root of $T^2-a_p( \l)T+p$, where $\hat \l=\varphi(\l)$ is the image of $\l$ under the Teichm\"uller character.

Since $\l=1$ is a singular case, we study the case when $\l=-1$, for which the corresponding Legendre curve admits complex multiplication. Let $p\equiv1 \pmod 4$ be prime,  then by \cite[(4.11)]{Greene},
$$p\cdot\, \hgp{\eta_2,}{\eta_2}{\varepsilon}{-1}=J(\eta_4,\eta_2)+J(\overline{\eta_4},\eta_2),$$
where $\eta_4$ is a primitive order 4 character of $\F_p^\times$.\footnote{When $p\equiv3 \pmod 4$,  $\hgp{\eta_2,}{\eta_2}{\varepsilon}{-1}=0$.} In the perspective of Dwork, the value \eqref{eq: 2F1-1/2} agrees with the unit root of $T^2 +(J(\eta_4,\eta_2)+J(\overline{\eta_4},\eta_2))T+p$, which is $\frac{\G_p\(\frac12\)\G_p\(\frac14\)}{\G_p\(\frac34\)}$ by the Gross-Koblitz formula. Using Lemma \ref{GaussToHyper}, we have the following.

\begin{prop}
For each  prime $p\equiv 1\pmod{4}$
\[
\pFq{2}{1}{\frac 12,\frac 12}{,1}{-1}_{\frac{p-1}2}  \equiv \frac{\G_p\(\frac12\)\G_p\(\frac14\)}{\G_p\(\frac34\)}=-\frac{\G_p\(\frac14\)}{\G_p\(\frac12\)\G_p\(\frac34\)}
\pmod{p}.
\]
\end{prop}

\begin{proof}

By Lemma \ref{AltTrunc} and Lemma \ref{GaussToHyper},  we obtain that

$$
  \pFq{2}{1}{\frac 12,\frac 12}{,1}{-1}_{\frac{p-1}2} \equiv \pFq{2}{1}{\frac 12,\frac 12}{,1}{-1}_{{p-1}}  \equiv -p\cdot \, \hgp{\eta_2,}{\eta_2}{\varepsilon}{-1} \pmod p.
$$
It remains to show
\begin{equation}\label{eq:3.1}
  p\cdot \, \hgp{\eta_2,}{\eta_2}{\varepsilon}{-1}  \equiv -\frac{\G_p\(\frac12\)\G_p\(\frac14\)}{\G_p\(\frac34\)} \pmod p.
\end{equation}
By the relations
$$
  p\cdot \, \hgp{\eta_2,}{\eta_2}{\varepsilon}{-1}  =J(\eta_4,\eta_2)+J(\overline{\eta_4},\eta_2)=\frac{g(\eta_2)\(g(\eta_4)^2+g(\overline{\eta_4})^2\)}{g(\overline{\eta_4})g(\eta_4)},
$$
using the Gross-Koblitz formula, we see that
\begin{align*}
  p\cdot \, \hgp{\eta_2,}{\eta_2}{\varepsilon}{-1} &= \frac{-\pi_p^{\frac{p-1}2}\G_p\(\frac12\)(\pi_p^{3\frac{p-1}2}\G_p\(\frac34\)^2+\pi_p^{\frac{p-1}2}\G_p\(\frac14\)^2)}
  {\pi_p^{{p-1}}\G_p\(\frac14\)\G_p\(\frac34\)}\\
   &=- \frac{\G_p\(\frac12\)(-p\G_p\(\frac34\)^2+\G_p\(\frac14\)^2)}{\G_p\(\frac14\)\G_p\(\frac34\)},
\end{align*}
which yields the result.
\end{proof}

Using a different technique via hypergeometric evaluation identities, one can prove the following
stronger result.  We will also outline this strategy here (for details, see \cite{LR}). First we deform the truncated hypergeometric series $p$-adically so that it becomes a  whole family of terminating series that can be written as a quotient of Pochhammer symbols $(a)_k$ via appropriate  hypergeometric evaluation formulas.  We further rewrite the quotient of Pochhammer symbols as a quotient of $p$-adic Gamma values using the functional equation of $p$-adic Gamma functions.  Then we use harmonic sums to analyze the terminating series on one side and use the Taylor expansion of $p$-adic Gamma functions on the other side. Now picking suitable members in the deformed family, we compare both sides to get a linear system which allows us to conclude the desired congruence.

 \begin{prop}\label{prop:9}
For any prime $p\equiv 1 \pmod 4$,
\[ \pFq{2}{1}{\frac 12,\frac12}{,1}{-1}_{\frac{p-1}2}\equiv  -\frac{\G_p(\frac 14)}{\G_p(\frac 12)\G_p(\frac 34)} \pmod{p^2}.\]
\end{prop}

\begin{remark}
Proposition \ref{prop:9} was first obtained by Coster and van Hamme \cite{Coster-vanHamme} using a refined version of formal group laws.
\end{remark}

\begin{proof}   We first recall a theorem of Kummer (see Corollary 3.1.2 of \cite{AAR}) which says that whenever $b$ is a negative integer,
\begin{equation}\label{Kummer}
\pFq{2}{1}{a,b}{,a-b+1}{-1}=\frac{\G(a-b+1)\G(a/2+1)}{\G(a+1)\G(a/2-b+1)}  =\frac{(a+1)_{-b}}{(a/2+1)_{-b}}.
\end{equation}
We now estimate the left hand side of the proposition statement, modulo $p^2$.  Observe that for any positive integer $1\leq k\leq \frac{p-1}2$, and $x,y\in \Z_p$,
\begin{align*}
\left (\frac 12+xp \right )_k &\equiv \left (\frac 12 \right )_k(1+2xp H_k^{(odd)}) \pmod {p^2}, \\
\left (1+yp \right )_k &\equiv \left (1 \right )_k(1+yp H_k) \pmod {p^2},
\end{align*}
where $\displaystyle H_k^{(odd)}:=\sum_{j=1}^{k}\frac1{2j-1}$ and $ \displaystyle H_k:=\sum_{j=1}^{k}\frac1{j}$ are harmonic sums.  Thus we have for any $x_1, x_2,y \in\Z_p$,
\begin{multline}\label{LeftSide}
\pFq{2}{1}{\frac 12+x_1p & \frac12+x_2p}{& 1+yp}{-1}_{\frac{p-1}2}\equiv 
 \pFq{2}{1}{\frac 12 & \frac12}{&1}{-1}_{\frac{p-1}2} + (x_1+x_2)Ap - yBp \pmod {p^2},
\end{multline}
where $\displaystyle  A=\sum_{k=0}^{\frac{p-1}2} \left ( \frac{(\frac 12)_k^2}{k!^2} \right ) (-1)^k \cdot 2 H_k^{(odd)}$, and  $\displaystyle  B=\sum_{k=0}^{\frac{p-1}2} \left ( \frac{(\frac 12)_k^2}{k!^2} \right )(-1)^k H_k$.

\noindent If  $b=\frac{1}{2} + x_2 p$ is a negative integer, then by  \eqref{Kummer} and the above analysis
\begin{equation}\label{combined}
\pFq{2}{1}{\frac 12,\frac12}{,1}{-1}_{\frac{p-1}2} + (x_1+x_2)Ap - (x_1-x_2)Bp \equiv \frac{\left(\frac 32+x_1p\right)_{-b}}{\left(\frac 54 + \frac{x_1p}{2}\right)_{-b}}  \pmod {p^2}.
\end{equation}

We now estimate the right hand side of the proposition, which can be also written in terms of the Pochhammer symbols.   One can use Lemma 17 of \cite{LR} to convert it to a quotient of $p$-adic Gamma function values. For example, if we let $x_1 = \frac12$,  $x_2=-\frac12$ in \eqref{combined} (thus $b=\frac{1-p}2$), the right hand side becomes
$$
\frac{\left(\frac{3+p}{2} \right)_{\frac{p-1}{2}}}{\left(\frac{5+p}{4}\right)_{\frac{p-1}{2}}}
= -\frac{\G_p(p)\G_p(\frac 14+\frac{p}4)}{\G_p(\frac12+\frac p2)\G_p(\frac 34+\frac {3p}4)}.
$$
By Proposition \ref{GpFacts},
\[
\G_p(\alpha+mp)\equiv \G_p(\alpha)[1+G_1(\alpha)mp] \pmod{p^2},
\]
and $G_1(\alpha)=G_1(1-\alpha)$. Thus,
\begin{multline*}
 -  \frac{\G_p(p)\G_p(\frac 14+\frac{p}4)}{\G_p(\frac12+\frac p2)\G_p(\frac 34+\frac {3p}4)} \equiv 
-\frac{\G_p(\frac 14)}{\G_p(\frac 12)\G_p(\frac 34)} \left[ 1+G_1(0)p-G_1\left(\frac 12\right)\frac{p}{2}-G_1\left(\frac 14 \right) \frac p2 \right] \pmod {p^2}.
\end{multline*}
Equating both sides in \eqref{combined} gives that
\begin{multline*}
\pFq{2}{1}{\frac 12,\frac12}{,1}{-1}_{\frac{p-1}2}  - Bp \equiv \\
-\frac{\G_p(\frac 14)}{\G_p(\frac 12)\G_p(\frac 34)} \left[1+G_1(0)p-G_1\left(\frac 12\right) \frac p2-G_1\left(\frac 14\right) \frac p2 \right] \pmod {p^2}.
\end{multline*}
Similarly, letting $x_1 = -3/2$,  $x_2=-1/2$, we get
\begin{multline*}
\pFq{2}{1}{\frac 12,\frac12}{,1}{-1}_{\frac{p-1}2}  -2A + Bp  \equiv \\
- \frac{\G_p(\frac 14)}{\G_p(\frac 12)\G_p(\frac 34)} \left[1-G_1(0)p+G_1\left(\frac 12\right) \frac {3p}2-G_1\left(\frac 14\right) \frac {p}2 \right] \pmod{p^2}.
\end{multline*}
Letting $x_1 = -1/2$,  $x_2=-3/2$ (here $b=\frac{1-3p}{2}$ but the series terminates at the $\frac{p-1}2$th term as $a=-\frac{p-1}2$),
\begin{multline*}
\pFq{2}{1}{\frac 12,\frac12}{,1}{-1}_{\frac{p-1}2}  -2A - Bp  \equiv \\
-\frac{\G_p(\frac 14)}{\G_p(\frac 12)\G_p(\frac 34)} \left[1+G_1(0)p+G_1\left(\frac 12\right) \frac p2-G_1\left(\frac 14\right) \frac {3p}2 \right] \pmod{p^2}.
\end{multline*}
By summing the first two and last two congruences and comparing, we arrive at the proposition, in light of Lemma \ref{AltTrunc}.
\end{proof}

Based on Lemma \ref{GaussToHyper}, we observe the following numerically which is a companion form of the congruence above.
\begin{conj}
For any prime $p\equiv1 \pmod 4$,
\begin{multline*}
-\frac{\G_p(\frac 14)}{\G_p(\frac 12)\G_p(\frac 34)}\equiv \sum_{k=0}^{p-1}  \(\frac{(\frac12)_k}{k!} \)^2(-1)^k\overset{?}{\equiv} \\
 \sum_{k=0}^{p-1}  \(\frac{k!}{(\frac32)_k} \)^2p^2(-1)^k\equiv \sum_{k=\frac{p-1}2}^{p-1}  \(\frac{k!}{(\frac32)_k} \)^2p^2(-1)^k
 \pmod {p^2}.
\end{multline*}
\end{conj}

\section{Proofs of Theorem \ref{thm:2} and Lemma \ref{GaussToHyper}}\label{ss:proofofthm4}

\subsection{Proof of Theorem \ref{thm:2}}
We first prove Theorem \ref{thm:2} which counts points on $C_{n,\l}$ over finite fields $\fq$ in terms of Gaussian hypergeometric functions.  We begin with a lemma.

\begin{lemma}\label{Gaussian_lemma}
Let $n\geq2$ be an integer and let $q$ be a prime power with $q=p^e\equiv 1 \pmod{n}$. Suppose $\eta_n\in\fqhat$ is a character of order $n$ and $\varepsilon$ denotes the trivial character. If $k\in\{1,\dots, n-1\}$, then

\begin{multline*}
\sum_{x_i\in\fq} \eta_n^k((x_1\cdots x_{n-1})^{n-1}(1-x_1)\cdots(1-x_{n-1})(x_1-\lambda x_2\cdots x_{n-1}))\\ = q^{n-1} \cdot\,  _{n}F_{n-1} \left(
\begin{matrix}
\eta_n^{n-k}, & \eta_n^{n-k}, & \dots, & \eta_n^{n-k} \\
     & \varepsilon, & \dots, & \varepsilon \\
\end{matrix}
; \lambda
\right)_q.
\end{multline*}
\end{lemma}

\begin{proof}
We apply Theorem \ref{Greene-3.13} $(n-3)$ times, noting that $\overline{\eta_n^{n-k}}=\eta_n^k$.  This gives

\begin{align*}
_{n}F_{n-1} \left(
\begin{matrix}
\eta_n^{n-k}, & \eta_n^{n-k}, & \dots, & \eta_n^{n-k} \\
     & \varepsilon, & \dots, & \varepsilon \\
\end{matrix}
; \lambda
\right)_q \\
&\hspace*{-1.5in}= \frac{(\eta_n^{n-k}(-1))^{n-3}}{q^{n-3}} \sum_{x_2,\dots, x_{n-2}\in\fq} \, _3F_2 \left(
\begin{matrix}
\eta_n^{n-k}, & \eta_n^{n-k}, & \eta_n^{n-k} \\
     & \varepsilon, & \varepsilon \\
\end{matrix}
; \lambda x_2\cdots x_{n-2}
\right)_q \\
&\hspace*{-1.25in}\cdot \eta_n^{n-k}(x_2\cdots x_{n-2}) \eta_n^k((1-x_2)\cdots (1-x_{n-2})).
\end{align*}

We next apply Corollary 3.14(ii) of \cite{Greene} to the $_3F_2$ function to obtain

\begin{multline*}
_{n}F_{n-1} \left(
\begin{matrix}
\eta_n^{n-k}, & \eta_n^{n-k}, & \dots, & \eta_n^{n-k} \\
     & \varepsilon, & \dots, & \varepsilon\\
\end{matrix}
; \lambda
\right)_q = \frac{(\eta_n^{n-k}(-1))^{n-2}}{q^{n-1}} \\
\cdot \sum_{x_1,\dots,x_{n-1}\in\fq} \eta_n^{n-k}(x_1\cdots x_{n-1})\eta_n^k((1-x_1)\cdots (1-x_{n-1})(x_1-\lambda x_2\cdots x_{n-1})).
\end{multline*}

\noindent Since $(\eta_n^{n-k}(-1))^{n-2}=\eta_n^{(n-k)(n-2)}(-1)=1$ and $$\eta_n^k((x_1\cdots x_{n-1})^{n-1})=\eta_n^{kn-k}(x_1\cdots x_{n-1})= \eta_n^{n-k}(x_1 \cdots x_{n-1}),$$ we have the result.

\end{proof} 

\noindent{We are now able to prove Theorem \ref{thm:2}. \\

\noindent {\it Proof of Theorem \ref{thm:2}.}
For convenience, we denote $$f(x_1,\dots, x_{n-1},\lambda)=(x_1\cdots x_{n-1})^{n-1}(1-x_1)\cdots(1-x_{n-1})(x_1-\lambda x_2\cdots x_{n-1}).$$
Then
\begin{align*}
\#C_{n,\lambda}(\fq)&=1+\sum_{x_i\in\fq} \#\{y\in\fq : y^n=f(x_1,\dots,x_{n-1},\lambda)\}\\
&= 1+ \sum_{\substack{x_i\in\fq \\  f(x_1,\dots,x_{n-1},\lambda)\neq 0}} \#\{y\in\fq : y^n=f(x_1,\dots,x_{n-1},\lambda)\} \\ \\
&\hspace*{0.5in} + \# \{ (x_1,\dots,x_{n-1})\in\mathbb{F}_q^{n-1} : f(x_1,\dots,x_{n-1},\lambda)=0\}.
\end{align*}
Using Prop. 8.1.5 in \cite{IR90}, we rewrite the first sum to see
\begin{align*}
\#C_{n,\lambda}(\fq) &= 1+\sum_{x_i\in\fq}\sum_{i=0}^{n-1} \eta_n^i(f(x_1,\dots,x_{n-1},\lambda))\\
&\hspace*{0.5in} + \# \{ (x_1,\dots,x_{n-1})\in\mathbb{F}_q^{n-1} : f(x_1,\dots,x_{n-1},\lambda)=0\}\\
&=1+q^{n-1}+\sum_{x_i\in\fq} \sum_{i=1}^{n-1} \eta_n^i(f(x_1,\dots,x_{n-1},\lambda)),
\end{align*}
since $$\varepsilon(f(x_1,\dots,x_{n-1},\lambda))+ \# \{ (x_1,\dots,x_{n-1})\in\mathbb{F}_q^{n-1} : f(x_1,\dots,x_{n-1},\lambda)=0\}=q^{n-1}.$$
Finally, we have
\begin{align*}
\#C_{n,\lambda}(\fq)&=1+q^{n-1}+ \sum_{i=1}^{n-1} \sum_{x_i\in\fq} \eta_n^i(f(x_1,\dots,x_{n-1},\lambda))\\
& = 1+q^{n-1} +q^{n-1} \, \sum_{i=1}^{n-1} \, _{n}F_{n-1} \left(
\begin{matrix}
\eta_n^{n-i}, & \eta_n^{n-i}, & \dots, & \eta_n^{n-i} \\
     & \varepsilon, & \dots, & \varepsilon \\
\end{matrix}
; \lambda
\right)_q
\end{align*}
by Lemma \ref{Gaussian_lemma}, which gives the result.\\

\subsection{Proof of Lemma \ref{GaussToHyper}}
To prove Lemma \ref{GaussToHyper}, which relates certain Gaussian hypergeometric functions to truncated hypergeometric series, we first give a lemma, which analyzes the Gaussian hypergeometric functions modulo a power of $p$.
\begin{lemma}\label{lem: GHF-Gp}
Let $n$ be a positive integer, and $p\equiv 1 \pmod n$ prime. Let $\varphi$ denote the Teichm\"uller character, and $\eta_n$ a character of order $n$ on $\F_p$ corresponding to $\varphi^{\frac{p-1}nj}$, for some $0<j<p-1$. Then
\begin{multline}\label{eq:2.3}
  p^{r-1}{} _{r}F_{r-1}\left(
\begin{matrix}
\overline{\eta_n}, & \overline{\eta_n}, & \cdots, & \overline{\eta_n} \\
   & \varepsilon, & \cdots, & \varepsilon \\
\end{matrix}
\, ; x
\right)_p \equiv \\
 \frac{1}{p-1}\(\sum_{k=(p-1)(\frac{n-j}{n})}^{p-2} \(\frac{\Gamma_p\(\frac k{p-1}-\frac{n-j}n\)}{\Gamma_p\(\frac k{p-1}\)\Gamma_p\(\frac jn\)} \)^r\overline\varphi^k(x)+(-1)^r\)
 \pmod{p^r},
\end{multline}

\begin{multline}\label{eq:2.4}
 p^{r-1}{} _{r}F_{r-1}\left(
\begin{matrix}
\eta_n, & \eta_n, & \cdots, & \eta_n \\
   & \varepsilon, & \cdots, & \varepsilon \\
\end{matrix}
\, ; x
\right)_p \equiv \\
 \frac{(-1)^r}{p-1}\(\sum_{k=0}^{(p-1)(\frac{n-j}{n})-1} \frac{\Gamma_p\(\frac k{p-1}\)^r\Gamma_p\(\frac jn\)^r}{\Gamma_p\(\frac jn+\frac k{p-1}\)^r} \overline\varphi^k(x)+\overline{\eta_n}((-1)^rx)\)  \pmod{p^r}.
\end{multline}

\end{lemma}

\begin{proof}
  For $x\in\F_p^\times$,  it follows from the definition that
\begin{multline*}
  {} _{r}F_{r-1} \left(
\begin{matrix}
\eta_n, & \eta_n, & \cdots, & \eta_n \\
   & \varepsilon, & \cdots, & \varepsilon \\
\end{matrix}
\, ; x
\right)_p= 
\frac p{p-1}\sum_{\chi\in \widehat{\F_p^\times}}\CC{\eta_n\chi}\chi^r \chi(x)=\frac {p^{1-r}}{p-1}\sum_{\chi}J(\eta_n\chi,\overline\chi)^r\chi^r(-1)\chi(x),
\end{multline*}
and also that
\begin{multline*}
  J(\overline{\eta_n}\chi,\overline\chi)=
  \begin{cases}
     \frac{\chi(-1)pg(\overline{\eta_n}\chi)}{g(\overline{\eta_n})g(\chi)},& \mbox{ if } \chi\neq \varepsilon,\\
     -1,& \mbox{ if } \chi=\varepsilon,
  \end{cases} 
\mbox{ while, }
   J(\eta_n\chi,\overline\chi)=
  \begin{cases}
     \frac{\chi(-1)g(\overline{\eta_n})g(\overline\chi)}{g(\overline{\eta_n}\overline\chi)},& \mbox{ if } \chi\neq \overline{\eta_n},\\
     -1,& \mbox{ if } \chi=\overline{\eta_n}.
  \end{cases}
\end{multline*}
Thus, we have that
\begin{multline*}
\sum_{\chi}J(\overline{\eta_n}\chi,\overline\chi)^r\chi^r(-1)\chi(x)=p^r\sum_{\chi\neq \varepsilon} \frac{g(\overline{\eta_n}\chi)^r}{g(\overline{\eta_n})^rg(\chi)^r}\chi(x)+(-1)^r\\
  =p^r\sum_{k=1}^{(p-1)(\frac{n-j}{n})-1} \frac{g(\overline\varphi^{\frac{p-1}nj+k})^r}{g(\overline\varphi^{\frac{p-1}nj})^rg(\overline\varphi^{k})^r}\overline\varphi^k(x) 
  +p^r\sum_{k=(p-1)(\frac{n-j}{n})}^{p-2} \frac{g(\overline\varphi^{k-(p-1)\frac{n-j}n})^r}{g(\overline\varphi^{\frac{p-1}nj})^rg(\overline\varphi^{k})^r}\overline\varphi^k(x)+\({-1}\)^r\\
   =p^r\sum_{k=1}^{(p-1)(\frac{n-j}{n})-1}(-1)^r \(\frac{\Gamma_p\(\frac jn+\frac k{p-1}\)}{\Gamma_p\(\frac k{p-1}\)\Gamma_p\(\frac jn\)} \)^r\overline\varphi^k(x)\\
  +\sum_{k=(p-1)(\frac{n-j}{n})}^{p-2} \(\frac{\Gamma_p\(\frac k{p-1}-\frac{n-j}n\)}{\Gamma_p\(\frac k{p-1}\)\Gamma_p\(\frac jn\)} \)^r\overline\varphi^k(x)+(-1)^r.
\end{multline*}

Similarly,
\begin{multline*}
   \sum_{\chi} J(\eta_n\chi,\overline\chi)^r\chi^r(-1)\chi(x)=\sum_{\chi\neq \overline{\eta_n}} \frac{g(\overline{\eta_n}\chi)^r}{g(\overline{\eta_n})^rg(\overline\chi)^r}\chi(x)+(-1)^r\overline{\eta_n}^r(-1)\overline{\eta_n}(x) \\
=(-1)^r\sum_{k=0}^{(p-1)(\frac{n-j}{n})-1} \(\frac{\Gamma_p\(\frac k{p-1}\)\Gamma_p\(\frac jn\)}{\Gamma_p\(\frac jn+\frac k{p-1}\)} \)^r\overline\varphi^k(x)
+ (-1)^r\overline{\eta_n}^r(-1)\overline{\eta_n}(x)\\
+ p^r\sum_{k=(p-1)(\frac{n-j}{n})+1}^{p-2} \(\frac{\Gamma_p\(\frac k{p-1}\)\Gamma_p\(\frac jn\)}{\Gamma_p\(\frac k{p-1}-\frac{n-j}n\)} \)^r\overline\varphi^k(x).
\end{multline*}
\end{proof} 

\noindent We are now able to prove Lemma \ref{GaussToHyper}.\\

\noindent {\it Proof of Lemma \ref{GaussToHyper}.}
For the first congruence, we use \eqref{eq:2.4}.  By Proposition \ref{GpFacts}, when $k<p$ we have that $\G_p(-k)=\G_p(-k)/ \G_p(0)=1/k!$ and $\G_p(\frac{k}{p-1})\equiv \G_p(-k) \pmod p$. Similarly,
$$\frac{\G_p(\frac jn)}{\G_p(\frac jn+\frac k{p-1})}\equiv \frac{\G_p(\frac jn)}{\G_p(\frac jn-k)}\pmod p$$
and
$$\frac{\G_p(\frac jn)}{\G_p(\frac jn-k)}=\left(1-\frac jn\right)\cdots\left( k-\frac jn\right)=\left(1-\frac jn\right)_k.$$
When $k=(p-1)\(\frac{n-j}n\)$, we have $\overline{\eta_n}(\pm(-1)^r)\equiv  \frac{(1-\frac jn)_k^r}{k!^r} (\pm1)^k  \pmod p$. Thus the first claim follows.

For the second claim, we consider \eqref{eq:2.3} and use a similar argument. Notice that
$$\frac{\G_p(\frac{k}{p-1}-\frac{n- j}n)}{\G_p(\frac jn)}\equiv \frac{\G_p(-k-\frac{n- j}n)}{\G_p(\frac jn)}=\frac{-p(\frac{n-j}n)}{(1-\frac jn)(2-\frac jn)\cdots (k-\frac jn)} \pmod p,$$
and there is no $-p(\frac{n-j}n)$ in the numerator since for $(p-1)(\frac{n-j}n) \leq k \leq p-2$ the denominator, $(1-\frac jn)_{k+1}$, will contain a multiple of $p$, which is $p(\frac {n-j}n)$. This term will not show up in the quotient of $p$-adic Gamma values and the corresponding term is $-1$ by the functional equation of $\G_p(\cdot)$. Thus
$$\frac{\G_p(\frac{k}{p-1}-\frac{n- j}n)}{\G_p(\frac jn)}\equiv -\frac{p}{(2-\frac jn)_k} \pmod p,$$
which concludes the proof of the second claim.\\

\section{The proof of Theorem \ref{GaussianTheorem}}\label{Gauss}

\noindent The following proposition establishes case (1) of Theorem \ref{GaussianTheorem}.

\begin{prop}
Let $q=p^e\equiv 1 \pmod 3$ be a prime power and let $\eta_3$ be a character of order 3 in $\fqhat$.  Then
$$ q^2\, \cdot \hgthree{\eta_3}{\eta_3}{\eta_3}{\varepsilon}{\varepsilon}{1}_q = J(\eta_3,\eta_3)^2-J(\eta_3^2,\eta_3^2).$$
\end{prop}

\begin{proof}
Beginning with Theorem 4.35 in \cite{Greene}, we have

$$q^2 \,\cdot \hgthree{\eta_3}{\eta_3}{\eta_3}{\varepsilon}{\varepsilon}{1}_q = q^2 \eta_3^2(-1) \binom{\eta_3}{\eta_3^2}\binom{\eta_3}{\eta_3^2} -q\eta_3(-1)\binom{\eta_3^2}{\eta_3}$$
Then, since $\binom{\eta_3}{\eta_3^2} =\frac{1}{q} J(\eta_3,\eta_3)$ and $\binom{\eta_3^2}{\eta_3} =\frac{\eta_3(-1)}{q} J(\eta_3^2,\eta_3^2),$ we get the result.

\end{proof} 

\noindent We now restate an equivalent form of case (2) of Theorem \ref{GaussianTheorem}.

\begin{theorem} Let $q=p^e\equiv 1\pmod{4}$ be a prime power and let $\eta_4$ and $\eta_2$ be characters of order 4 and 2, respectively, in $\fqhat$. Then

$$\sum_{x,y,z\in \F_q}\eta_4(x^3y^3z^3(1-x)(1-y)(1-z)(x-yz))
=J(\overline{\eta_4},\eta_2)^3+qJ(\overline{\eta_4},\eta_2)-J({\eta_4},\eta_2)^2. $$

Equivalently, we have
$$
q^3\cdot\, _4F_3\(\begin{matrix}
\overline{\eta_4}, & \overline{\eta_4},&\overline{\eta_4},&\overline{\eta_4}  \\
   & \varepsilon,  & \varepsilon, & \varepsilon\\
\end{matrix};1\)_q
=J(\overline{\eta_4},\eta_2)^3+qJ(\overline{\eta_4},\eta_2)-J({\eta_4},\eta_2)^2.
$$
\end{theorem}

As mentioned in the introduction, this result says that a 3-dimensional Galois representation of $G_{\Q(\sqrt{-1})}:=\Gal(\overline \Q/\Q(\sqrt{-1}))$ arising from $C_{4,1}$  is isomorphic to a direct sum of three Gr\"ossencharacters. The proof contains two steps using totally different approaches. We first establish the case for $q\equiv 1\pmod 8$ using results of Greene and McCarthy. This is equivalent to proving the two Galois representations are isomorphic when they are restricted to the subgroup  $G_{\Q(\sqrt{-1},\sqrt{2})}:=\Gal(\overline \Q/\Q(\sqrt{-1},\sqrt{2}))$. In the second part of the proof, we use representation theory to draw the final conclusion.
We now start with the case when $q=p^e\equiv 1 \pmod{8}$ is a prime power, we prove this case via the series of results below.  The Lemma below evaluates two modified Gaussian hypergeometric functions.

\begin{lemma}\label{lem:23}
Let $q=p^e\equiv 1 \pmod{8}$ be a prime power. Then
\begin{multline*}
 {}_4F_3\(\begin{matrix}
\overline{\eta_8}, &{\eta_4},&{\eta_4}, &\eta_8  \\
   & \eta_8,  & \eta_8^3, & \varepsilon\\
\end{matrix};1\)_q^\ast=\\
\frac{1}{J(\eta_4,\overline{\eta_8}^3)}\(J(\eta_8,\bet4)-\et8(-1)J(\overline{\eta_8},\overline{\eta_8})J(\et2,\et4)+\frac{J(\et8,\bet4)^3}q\),
\end{multline*}
\begin{multline*}
 {}_4F_3\(\begin{matrix}
{\eta_8^3}, &{\eta_4},&{\eta_4}, &\eta_8  \\
   & \overline{\eta_8}^3,  & \eta_8^3, & \varepsilon\\
\end{matrix};1\)_q^\ast= \\
\frac{1}{J(\eta_4,{\eta}_8^3)} \(J(\eta_8^3,\bet4)-\et8(-1)J({\eta}_8,\overline{\eta_8}^3)J(\et2,\et4)+\frac{J(\bet8,\bet4)J(\bet4,\et8^3)^2}q\).
\end{multline*}
\end{lemma}
\begin{proof}
  Firstly, we obtain that
$$
 {}_4F_3\(\begin{matrix}
\overline{\eta_8}, &{\eta_4},&{\eta_4}, &\eta_8  \\
   & \eta_8,  & \eta_8^3, & \varepsilon\\
\end{matrix};1\)_q^\ast =-q\(\CC{\et4}{\et8}\CC{\et4}{\et8^3}\)^{-1}
 {}_4F_3\(\begin{matrix}
\overline{\eta_8}, &{\eta_4},&{\eta_4}, &\eta_8  \\
   & \eta_8,  & \eta_8^3, & \varepsilon\\
\end{matrix};1\)_q.
$$
Using the transformations Theorem 3.15(iv), Theorem 4.35 in \cite{Greene}, and the fact $\eta_4(-1)=1$ when $q\equiv 1 \pmod 8$, we have
\begin{multline*}
  {}_4F_3\(\begin{matrix}
\overline{\eta_8}, &{\eta_4},&{\eta_4}, &\eta_8  \\
   & \eta_8,  & \eta_8^3, & \varepsilon\\
\end{matrix};1\)_q ={}_4F_3\(\begin{matrix} {\eta}_8, &{\eta_4},&{\eta_4}, &\bet8  \\ & \eta_8,  & \eta_8^3, & \varepsilon\\ \end{matrix};1\)_q \\
=\CC{\et4}{\et8}{}_3F_2\(\begin{matrix}
{\eta_4},&{\eta_4}, &\bet8  \\ & \eta_8^3, & \varepsilon\\ \end{matrix};1\)_q 
-\frac{\et8(-1)}q \CC{\et8}{\et4}\CC{\bet4}{\bet8} \\
=\CC{\et4}{\et8}\({\et8(-1)}\CC{\et8}{\bet8}\CC{\et4}{\et2}-\frac{{\et8(-1)}}q\CC{\et8}{\et4}\)-\frac1q \CC{\et8}{\et4}^2.
\end{multline*}
Therefore, we can conclude that
\begin{align*}
 {}_4F_3&\(\begin{matrix}
\overline{\eta_8}, &{\eta_4},&{\eta_4}, &\eta_8  \\
   & \eta_8,  & \eta_8^3, & \varepsilon\\
\end{matrix};1\)_q^\ast\\
&=\frac{1}{J(\eta_4,\overline{\eta_8}^3)}\(J(\eta_8,\bet4)-{\et8(-1)}J(\overline{\eta_8},\overline{\eta_8})J(\et2,\et4)+\frac{J(\et8,\bet4)^3}q\).
\end{align*}
Likewise, we have the equality
\begin{align*}
{}_4F_3&\(\begin{matrix}
{\eta_8^3}, &{\eta_4},&{\eta_4}, &\eta_8  \\
   & \overline{\eta_8}^3,  & \eta_8^3, & \varepsilon\\
\end{matrix};1\)_q^\ast = \\
&\frac{1}{J(\eta_4,{\eta}_8^3)}\(J(\eta_8^3,\bet4)-{\et8(-1)}J({\eta}_8,\overline{\eta_8}^3)J(\et2,\et4)+\frac{J(\bet8,\bet4)J(\bet4,\et8^3)^2}q\).
\end{align*}

\end{proof} 
Next, we relate our target to a modified Gaussian $_5F_4$ hypergeometric function.

\begin{prop}
  Let $q=p^e\equiv 1 \pmod{8}$ be a prime power, and $\eta_8$ a character of order $8$ in $\fqhat$ with $\eta_8^2=\eta_4$ . Then
$$
q^4\cdot{}_4F_3\(\begin{matrix}
{\eta_4}, &{\eta_4},&{\eta_4},&{\eta_4}  \\
   & \varepsilon,  & \varepsilon, & \varepsilon\\
\end{matrix};1\)_q
=J({\eta_8},\eta_8)^4-q\cdot\,{}_5F_4\(\begin{matrix}
{\eta_4}, &{\eta_4},&{\eta_4},&{\eta_4}, &\eta_8  \\
   & \varepsilon,  & \varepsilon, & \varepsilon,&\eta_8\\
\end{matrix};1\)_q^\ast.
$$
\end{prop}
\begin{proof}
Comparing the definitions of finite field hypergeometric functions given by Greene and McCarthy, one can find
\begin{align*}
  {}_5F_4&\(\begin{matrix}
{\eta_4}, &{\eta_4},&{\eta_4},&{\eta_4}, &\eta_8  \\
   & \varepsilon,  & \varepsilon, & \varepsilon,&\eta_8\\
\end{matrix};1\)_q\\
&=\frac1{q^3}{}_5F_4\(\begin{matrix}
{\eta_4}, &{\eta_4},&{\eta_4},&{\eta_4}, &\eta_8  \\
   & \varepsilon,  & \varepsilon, & \varepsilon,&\eta_8\\
\end{matrix};1\)_q^\ast + \(1-\frac 1q\){}_4F_3\(\begin{matrix}
{\eta_4}, &{\eta_4},&{\eta_4},&{\eta_4}  \\
   & \varepsilon,  & \varepsilon, & \varepsilon\\
\end{matrix};1\)_q.
\end{align*}
On the other hand, by \cite[Theorem 3.15(ii)]{Greene}, the Greene's ${}_5F_4$ function is equal to
$$
   -\frac1q{}_4F_3\(\begin{matrix}
{\eta_4}, &{\eta_4},&{\eta_4},&{\eta_4}  \\
   & \varepsilon,  & \varepsilon, & \varepsilon\\
\end{matrix};1\)_q+\frac{J(\et8,\et8)^4}{q^4}.
$$
These lead to the desired result.
\end{proof} 
We now evaluate the $_5F_4$ modified Gaussian hypergeometric function using Theorem \ref{McCarthy5F4} which is due to McCarthy.
\begin{prop}\label{prop: 5F4}
Let $q=p^e\equiv 1 \pmod{8}$ be a prime power. Then
$$
 {}_5F_4\(\begin{matrix}
{\eta_4}, &{\eta_4},&{\eta_4},&{\eta_4}, &\eta_8  \\
   & \varepsilon,  & \varepsilon, & \varepsilon,&\eta_8\\
\end{matrix};1\)_q^\ast= \frac{J(\et8,\et8)^4}{q}-qJ(\et4,\et2)-J(\et2,\et4)^3+J(\et2,\bet4)^2.
$$
\end{prop}

\begin{proof}
According to Theorem \ref{McCarthy5F4}, we can deduce that
\begin{align*}
    {}_5F_4&\(\begin{matrix}
{\eta_4}, &{\eta_4},&{\eta_4},&{\eta_4}, &\eta_8  \\
   & \varepsilon,  & \varepsilon, & \varepsilon,&\eta_8\\
\end{matrix};1\)_q^\ast= \frac{q}{J(\bet8,\et4)}{}_2F_1\(\begin{matrix}
{\eta_4},&{\eta_4}   \\
        & \varepsilon\\
\end{matrix};-1\)_q^\ast\\
&+q\frac{J(\et8,\et8)}{J(\et8^3,\bet8)}\({}_4F_3\(\begin{matrix}
\overline{\eta_8} &{\eta_4}&{\eta_4} &\eta_8  \\
   & \eta_8,  & \eta_8^3, & \varepsilon\\
\end{matrix};1\)_q^\ast+{}_4F_3\(\begin{matrix}
{\eta_8^3}, &{\eta_4},&{\eta_4}, &\eta_8  \\
   & \overline{\eta_8}^3,  & \eta_8^3, & \varepsilon\\
\end{matrix};1\)_q^\ast\).
\end{align*}
We now apply Lemma \ref{lem:23} and the following fact which arises from Theorem 1.9 of \cite{McCarthy}
$$
  {}_2F_1\(\begin{matrix}
{\eta_4}&{\eta_4}   \\
        & \varepsilon\\
\end{matrix};1\)_q^\ast
=-J(\et8,\bet4)-J(\bet4,\bet8^3).
$$
To simplify the formulas, we recall the facts that if $q=p^e\equiv 1 \pmod 8$, we have $\et2(2)=\et4(-1)=1$
and the identities

\begin{align*}
  J(\et8,\et4)&=\bet8^3(-4)J(\et2,\et4),\\
  J(\bet4,\et8^3)&=\et8(-1)J(\bet8,\et8^3)=\bet8(-4)J(\et2,\bet4),\\
  J(\et8,\bet4)&=J(\et4,\et8^3)=\et8(-1)J(\et8^3,\et8^3)=\et8(-1)J(\et8,\et8).
\end{align*}

For more details on Jacobi sums, please see \cite[Chapter 3]{Berndt-Evans-Williams}.
\end{proof}

We now conclude with the second case using representation theory. We first describe the  Gr\"ossencharacters corresponding to the Jacobi sums 
\begin{multline*}
-J(\overline{\eta_4},\eta_2)^3-qJ(\overline{\eta_4},\eta_2)+J({\eta_4},\eta_2)^2= 
-J(\overline{\eta_4},\eta_2)^3-J(\overline{\eta_4},\eta_2)J(\overline{\eta_4},\eta_2)^2+J({\eta_4},\eta_2)^2
\end{multline*}
and then we use properties of induced representations to draw the final conclusion.
\begin{proof}[The proof of the case when prime powers $q\equiv 5 \pmod 8$.]

For each prime ideal  $\mathfrak p$ prime to $(4)$ in $\Z[\sqrt{-1}]$,  let $q$ be the norm of $\mathfrak p$. Then $q\equiv1 \mod \, 4$, and let $\psi_{\mathfrak p}$ be a homomorphism of $\Z[\sqrt{-1}]/\mathfrak p$ to the order 4 multiplicative group $\langle \sqrt{-1} \rangle$ such that $\psi_{\mathfrak p}(x)\equiv x^{\frac{q-1}4} \mod \mathfrak p$ for each $x\in \Z[\sqrt{-1}]$.  The map that assigns $\displaystyle
-\sum_{x \mod \mathfrak p}\psi_{\mathfrak p}(x)\psi_{\mathfrak p}^2(1-x)=-J(\psi_{\mathfrak p},\psi_{\mathfrak p}^2)
$ to  $\frak p$
extends to a Hecke (or Gr\"ossencharacter) character $\psi$ of $G_{\Q(\sqrt{-1})}$ (see \cite{Weil} by  Weil). In particular, $\psi$ is of conductor $\((1+\sqrt{-1})^4\)=(4)$ and infinity-type $[1,0]$, which is corresponding to the elliptic curve with complex multiplication which has conductor $64$. Explicitly, for any $a+b\sqrt{-1}\in \Z[\sqrt{-1}]$,  $\psi(a+b\sqrt{-1})=(-1)^b(a+b\sqrt{-1})\chi_1(a+b\sqrt{-1})$, where
$$
  \chi_1(a+b\sqrt{-1})=
  \begin{cases}
     (-1)^{\frac{a+b-1}2}\sqrt{-1},&\, \mbox{ if } a\equiv 0 \pmod{2}, b\equiv 1 \pmod{2},\\
    (-1)^{\frac{a+b-1}2},&\, \mbox{ if } a\equiv 1\pmod{2}, b\equiv 0 \pmod{2},\\
      0,&\, \mbox{ otherwise.}
  \end{cases}
$$
Here, we remark that the unit group of $\Bbb Z[\sqrt{-1}]/(4)$ has order $8$ and it is  generated by $\sqrt{-1}$, $-1+2\sqrt{-1}$. The Dirichlet character $\chi_1$ takes  values $\chi_1(\sqrt{-1})=\sqrt{-1}$, $\chi_1(-1+2\sqrt{-1})=1$.

By class field theory, $\psi$ corresponds to a character $\chi$ of $G_{\Q(\sqrt{-1})}$.
 Similar to the discussion in \cite{DFLST}, for each Frobenius class $\Frob_q\in G_{\Q(\sqrt{-1},\sqrt{2})}$ with $q\equiv 1\pmod 4$, 
\begin{multline*}
\sum_{x,y,z\in \F_q} \left
[-\eta_4\(x^3y^3z^3(1-x)(1-y)(1-z)(x-yz)\)- \overline{\eta_4}\(x^3y^3z^3(1-x)(1-y)(1-z)(x-yz)\)\right ]
\end{multline*}
coincides with the trace of $\Frob_{\frak p}$ under the 6-dimensional semisimple
representation $$\rho:=\text{Ind}_{G_{\Q(\sqrt{-1})}}^{G_\Q} \left
(\overline{\chi}^3\oplus (\overline{\chi}^2\otimes \chi)\oplus \chi^2
\right ).$$   Moreover,
$\rho|_{G_{\Q(\sqrt{-1})}}=\sigma\oplus \bar{\sigma}$, with the restriction $\sigma|_{
G_{\Q(\sqrt{-1},\sqrt{2})}}$ being isomorphic to the restriction of $\overline{\chi}^3\oplus
(\overline{\chi}^2\otimes \chi)\oplus \chi^2$ to $G_{\Q(\sqrt{-1},\sqrt{2})}$. As $
G_{\Q(\sqrt{-1},\sqrt{2})}$ is an index-2 subgroup of
$G_{\Q(\sqrt{-1})}$, by Clifford's result \cite{Clifford}, $\sigma$ is also direct sum
of the form
$$(\overline{\chi}^3 \otimes \varphi^{n_1})\oplus
(\overline{\chi}^2\otimes \chi  \otimes\varphi^{n_2})\oplus (\chi^2
\otimes\varphi^{n_3})$$ where $\varphi$ is
the order 2 character of $G_{\Q(\sqrt{-1})}$ with kernel $
G_{\Q(\sqrt{-1},\sqrt{2})}$, $n_1,n_2,n_3\in \{ 0,1\}$. From computing a few primes $p\equiv
5\pmod 8$, we determine that each $n_i=0$ and the claim  for
$p\equiv 5 \pmod 8$ thus follows.
\end{proof}

\section{Supercongruences}\label{Super}
We first prove Theorem \ref{thm:4a} using the technique outlined before Proposition \ref{prop:9}.
The initial idea of the proof is due to Zudilin, and uses the following particular case of the Karlsson--Minton formula \cite[eq.~(1.9.3)]{GR}: for any non-negative integers $m_1, \ldots, m_n$,
\begin{multline}\label{KM1}
\pFq{n+1}{n}{
-(m_1+\dots+m_n)& \, b_1+m_1&  \dots & b_n+m_n}{&
b_1 &\dots & b_n }{1}
\\
=(-1)^{m_1+\dots+m_n}\frac{(m_1+\dots+m_n)!}{(b_1)_{m_1}\dotsb(b_n)_{m_n}}.
\end{multline}
Note that when $n=2$, we can derive a different proof using a formula of Dixon. We present the proof below as it applies to all $n$.\\

\noindent{\it Proof of Theorem \ref{thm:4a}.}\\
Let $n\geq 3$, and $p\equiv 1 \pmod n$ be prime, which has to be odd.  Set $m=\frac{p-1}n$, and let $y$ be any integer. Letting $b_1=1+yp$, $b_2=\cdots =b_{n}=1$, and $m_1=\cdots =m_{n}=m$ in \eqref{KM1}, we get that
\begin{multline}\label{KM2}
\pFq{n+1}{n}{1-p&1+m+yp&1+m&\cdots& 1+m}{& 1+yp&1&\cdots&1}{1} \\ = \frac{(-1)^{p-1}(p-1)!}{(1+yp)_m(m!)^{n-1}}=\frac{(p-1)!}{(1+yp)_m(m!)^{n-1}}.
\end{multline}
Now we compare the left hand side with the left hand side of Theorem \ref{thm:4a}. We observe that
\begin{align*}
(1-p)_k & \equiv (1)_k -p\sum_{i=1}^k\frac{(1)_k}{i} \pmod{p^2}, \\
(1+yp)_k & \equiv (1)_k +yp\sum_{i=1}^k\frac{(1)_k}{i} \pmod{p^2},
\end{align*}
and so
\begin{equation}\label{(1+y)1}
\frac{(1-p)_k}{(1+yp)_k} \equiv 1 - (1+y)\left[ \sum_{i=1}^k i^{-1} \right]p \pmod{p^2}.
\end{equation}
Also,
\begin{multline*}
(1+m+yp)_k = \left(\left(1-\frac 1n\right) + \left(\frac 1n + y\right)p\right)_k \equiv \\
\left(1-\frac 1n\right)_k + \left(\frac 1n +y\right) \left[\sum_{i=1}^k\frac{(1-\frac 1n)_k}{(i-\frac 1n)}\right]p \pmod{p^2},
\end{multline*}
\begin{multline*}
(1+m)_k^{n-1}  \equiv \left(1-\frac 1n\right)_k^{n-1} + \left(\frac{n-1}n\right) \left(1-\frac 1n \right)_k^{n-2}\left[ \sum_{i=1}^k \frac{(1-\frac1n)_k}{(i-\frac 1n)} \right]p \pmod{p^2},
\end{multline*}
and so
\begin{multline}\label{(1+y)2}
(1+m+yp)_k(1+m)_k^{n-1} \equiv 
\left( 1-\frac 1n\right)_k^n\left(1 + (1+y) \left[\sum_{i=1}^k \left(i-\frac 1n\right)^{-1} \right] p\right) \pmod{p^2}.
\end{multline}
Thus \eqref{(1+y)1} and \eqref{(1+y)2} give that the left hand side of \eqref{KM2} can be written modulo $p^2$ as
\[
\sum_{k=0}^{p-1} \frac{\left(1-\frac 1n \right)_k^n}{(1)_k^{n-1}}\left( 1 + (1+y) \left[ \sum_{i=1}^k \left( \left(i - \frac 1n\right)^{-1} - i^{-1}\right)  \right] p \right) \pmod{p^2}.
\]

Using harmonic sums we conclude that there exists $A\in \Z_p$ such that
\begin{multline}\label{eq:6.2a} \pFq{n+1}{n}{1-p&1+m+yp& 1+m&\cdots & 1+m}{& 1+yp&1&\cdots&1}{1}\\ \equiv  \pFq{n}{n-1}{\frac{n-1}n&\frac{n-1}n&\cdots&\frac{n-1}n}{&1&\cdots&1}{1}_{p-1} \cdot\(1-A(y+1)p\) \pmod {p^2}.
\end{multline}
Using part c) of Proposition \ref{GpFacts}, we see that $a_0(\frac 1n)=\frac{1+(n-1)p}n=p-m$, and by part b),
$$\G_p(-m)=1/m! \; \mbox{ and } \; \G_p(-p)=-1/(p-1)!.$$
Also
\begin{multline*} 
\frac{1}{(1+yp)_m}= \frac{(-1)^m \G_p(1+yp)}{\G_p(1 + yp +\frac{p-1}{n})} = \frac{-(-1)^m \G_p(yp)}{\G_p(1 - \frac{1}{n} + (\frac{1}{n} + yp ))}= 
\G_p\left(\frac{1}{n} - \left(\frac{1}{n} + y\right)p\right)\G_p(yp).
\end{multline*}

We thus have
\begin{multline}\label{eq:6.3a}
\frac{(p-1)!}{(1+yp)_m(m!)^{n-1}} = -\frac{\G_p(\frac{1-p}n)^{n-1}\G_p(\frac 1n - (\frac 1n + y)p)\G_p(yp)}{\G_p(-p)}\\
\equiv- \G_p\(\frac 1n\)^n\(1+\(G_1(0)-G_1\(\frac 1n\)\)(1+y)p\) \pmod{p^2}.
\end{multline}
Finally, letting $y=-1$ in \eqref{eq:6.2a} and \eqref{eq:6.3a}, we see the desired congruence modulo $p^2$.\\

\noindent We now prove Theorem \ref{H4}.\\

\noindent \emph{Proof of Theorem \ref{H4}.}
Let $p$ be a prime such that $p \equiv 1\pmod{4}$.  We will use the following formula of Dougall (see Theorem 3.5.1 of \cite{AAR}) which says that if $2a+1=b+c+d+e-m$, then
\begin{multline}\label{Dougall}
\pFq{7}{6}{a& a/2 + 1& b&c&d&e&-m}{&a/2& 1+a-b& 1+a-c& 1+a-d& 1+a-e& 1+a+m}{1} \\
= \frac{(1+a)_m(1+a-b-c)_m(1+a-b-d)_m(1+a-c-d)_m}{(1+a-b)_m(1+a-c)_m(1+a-d)_m(1+a-b-c-d)_m}.
\end{multline}
Letting $a=1/4,b=5/8$, $c=1/8$, $d=(1+pu)/4$, $e=(1+(1-u)p)/4$, and $m=(p-1)/4)$, we have
$2a+1=b+c+d+e-m$ and thus we can use \eqref{Dougall}.  We first observe that the left hand side reduces to
\[
\pFq{4}{3}{\frac14 & \frac{1+pu}4 & \frac{1+(1-u)p}4 & \frac{1-p}4}{ &1-\frac{pu}4 &1+\frac{(u-1)p}4 &1+\frac p4}{1}
\]
after deleting three matching pairs of upper and lower parameters corresponding to $a/2 +1$, $b$, and $c$.
Writing the right hand side in terms of Gamma functions, we thus get that
\begin{multline*}
\pFq{4}{3}{\frac14 & \frac{1+pu}4 & \frac{1+(1-u)p}4 & \frac{1-p}4}{ &1-\frac{pu}4 &1+\frac{(u-1)p}4 &1+\frac p4}{1} \\
=\frac{\G(1+\frac p4)\G(\frac{1+p}4)\G(\frac 18+\frac{(1-u)p}{4})\G(\frac 58+\frac{(1-u)p}{4})\G(\frac 58)\G(\frac 98)\G(1-\frac{pu}4)\G(\frac{1-pu}4) }
{\G(\frac 54)\G(\frac12)\G(\frac38-\frac{pu}4)\G(\frac78-\frac{pu}4)\G(\frac38+\frac{p}4)\G(\frac78+\frac{p}4)\G(\frac{3+(1-u)p}4) \G(\frac{(1-u)p}4)}.
\end{multline*}
Using the duplication formula $\G(z)\G(z+\frac 12)=2^{1-2z}\G(\frac
12)\G(2z)$, we have
\[
\frac{\G(\frac 18+\frac{(1-u)p}{4})\G(\frac 58+\frac{(1-u)p}{4})\G(\frac 58)\G(\frac 98)}{\G(\frac38-\frac{pu}4)\G(\frac78-\frac{pu}4)\G(\frac38+\frac{p}4)\G(\frac78+\frac{p}4)}
=\frac{\G(\frac14+\frac{(1-u)p}{2})\G(\frac54)}{\G(\frac34-\frac{pu}2)\G(\frac34+\frac{p}2)}.
\]
Using Proposition \ref{GpFacts}, we can thus rewrite the right hand side as
\begin{multline*}
\frac{\G_p(\frac p4)\G_p(\frac{1+p}4)\G_p(\frac14+\frac{(1-u)p}{2})\G_p(-\frac{pu}4)\G_p(\frac{1-pu}4)}{\G_p(\frac12)\G_p(\frac34-\frac{pu}2)\G_p(\frac34+\frac{p}2)\G_p(\frac{3+(1-u)p}4)\G_p(\frac{(1-u)p}4)}\\
\equiv (-1)^{\frac{p-1}4}\G_p\left(\frac12\right)\G_p\left(\frac14\right)^6
\cdot \left(1-\frac{5(u^2-u+1)(G_1(\frac 14)^2-G_2(\frac 14))}{16}p^2 \right.\\
\left. +\frac{u(u-1)(G_1(0)^3-G_3(0)-21G_2(\frac 14)G_1(\frac 14)+7G_3(\frac 14)+14G_1(\frac 14)^3)}{128}p^3\right) \pmod{p^4}.
\end{multline*}
Meanwhile, we expand the left hand side using harmonic sums as in Proposition \ref{prop:9}.
So there exist $a_{k,i},b_{k,i}\in \Z_p$ such that modulo $p^4$ we have
\begin{multline*}
\pFq{4}{3}{\frac 14&\frac{1+pu}4&\frac{1+(1-u)p}4&\frac{1-p}4}{&1-\frac{pu}4&1+\frac{(u-1)p}4&1+\frac
p 4}{1}= \\
\sum_{k=0}^{\frac{p-1}4} \left( \frac{(\frac 14)_k^4}{k!^4} 
 \frac{ (1+a_{k,1} \frac{pu}4+a_{k,2}(\frac{pu}4)^2+a_{k,3}(\frac{pu}4)^3)(1+a_{k,1} \frac{(1-u)p}4+a_{k,2}(\frac{(1-u)p}4)^2+a_{k,3}(\frac{(1-u)p}4)^3)}{(1+b_{k,1}\frac{-pu}4+b_{k,2}(\frac{-pu}4)^2+b_{k,3}(\frac{-pu}4)^3)} \right. \\
\left. \cdot \frac{(1+a_{k,1}\frac{-p}4+a_{k,2}(\frac{-p}4)^2+a_{k,3}(\frac{-p}4)^3)}
{(1+b_{k,1} \frac{(u-1)p}4+b_{k,2}(\frac{(u-1)p}4)^2+b_{k,3}(\frac{(u-1)p}4)^3)(1+b_{k,1}\frac{p}4+b_{k,2}(\frac{p}4)^2+b_{k,3}(\frac{p}4)^3)} \right).
\end{multline*}
Note that if we collect coefficients of $p,p^2,p^3$ we get $0$,
$$
-\frac{(u^2-u+1)(a_{k,1}^2-2b_{k,2}-2a_{k,2}+b_{k,1}^2)}{16},
$$
and
$$
-\frac{u(u-1)(-3b_{k,3}+3a_{k,3}-b_{k,1}^3-3a_{k,1}a_{k,2}+3b_{k,1}b_{k,2}+a_{k,1}^3)}{64},
$$
respectively. \footnote{Comparing both sides, when we pick
$u=-\zeta_3$ where $\zeta_3$ be a primitive cubic root, we can derive
that the claim of the theorem holds modulo $p^3$.}

By collecting terms, we see there are $A_i,B_i\in \Z_p$ such that for all $u\in \Z_p$
\begin{multline*}
\pFq{4}{3}{\frac 14&\frac 14&\frac 14&\frac 14}{&1 &1 &1}{1}_{\frac{p-1}{4}}[1+A_2(u^2-u+1)p^2+A_3(u^2-1)p^3] \\
\equiv (-1)^{\frac{p-1}{4}}\G_p \left(\frac12\right) \G_p \left(\frac 14\right)^6 [1+B_2 (u^2-u+1) p^2 + B_3 (u^2-1) p^3] \pmod{p^4},
\end{multline*}
from which the claim of the theorem follows.

\section{Remarks}\label{Rem}
\subsection{Truncated hypergeometric series and noncongruence modular forms}
We first recall the following hypergeometric series transformation (see  \cite[(2.4.12)]{AAR})
\begin{equation}
 \pFq{3}{2}{-m,a,b}{,d,e}{1}=\frac{(e-a)_m}{(e)_m}\pFq{3}{2}{-m,a,d-b}{,d,a+1-m-e,}{1},
\end{equation}
which holds when $-m$ is a negative integer and both sides converge. Given an integer $n\ge 2$ and a prime $p\equiv 1\pmod n$, letting first $m=\frac{p-1}n$, $a=\frac{n-1+p}n$, $b=\frac 1n$, $d=e=1$, and then $m=\frac{(n-1)(p-1)}3$, $a=\frac{1+(n-1)p}n,$ $b=\frac 1n$, $d=e=1$ respectively, we derive the following supercongruence
\begin{equation}\label{eq:7.2}(-1)^{\frac{p-1}n}\pFq{3}{2}{\frac{n-1}n,\frac{n-1}n,\frac1n}{,1,1}{1}_{p-1}\equiv
\pFq{3}{2}{\frac1n,\frac1n,\frac{n-1}n}{,1,1}{1}_{p-1}\,
\pmod {p^2}.\end{equation} This was first observed and proved by McCarthy in a private communication via a different approach using the work of Mortenson \cite{Mortenson} for the case of $n=3$. Moreover, we would like to mention the following conjecture to demonstrate that truncated  hypergeometric series arise in many different settings including the theory of noncongruence modular forms.

\begin{conj}For any integer $n>1$ and prime $p\equiv 1 \pmod n$,
$$\pFq{3}{2}{\frac1n,\frac1n,\frac{n-1}n}{,1,1}{1}_{p-1}\equiv a_p(f_n(z)) \pmod {p^2},$$
where $a_p(f_n(z))$ is the $p$th coefficient of $f_n(z)=\sqrt[n]{E_1(z)^{n-1}E_2(z)}$ when expanded in terms of the local uniformizer $e^{2\pi i z/5n}$, and $E_1(z)$ and $E_2(z)$ are two explicit level 5 weight 3 Eisenstein series with coefficients in $\Z$ (see (17) and (18) of \cite{LLZ} or \cite[\S 3]{LL} for their expansions).
\end{conj}
In a series of papers \cite{LLZ, ALL, Long}, the third author and her collaborators studied the properties of these functions $f_n(z)$ which are weight 3 cusp forms for some finite index subgroups of $SL_2(\Z)$ that contain no principal congruence subgroups. For $n=3,4,6$,  the $p$th coefficients of $f_n(z)$  are shown to be related to the coefficients of classical Hecke or Hilbert modular forms.

In terms of Gaussian hypergeometric functions, we have
when $p\equiv 1\pmod n$ and $\eta_n$ any order $n$ character of $\F_p^\times$ then by work of Greene \cite{Greene}
\begin{equation}\hgthree{\eta_n}{\eta_n}{\overline {\eta_n}}{\varepsilon}{\varepsilon}{1}_p=\eta_n(-1) \hgthree{\eta_n}{\overline{\eta_n}}{\overline {\eta_n}}{\varepsilon}{\varepsilon}{1}_p.
\end{equation}

\subsection{Other observations}
 We conclude with some patterns observed from numerical data.  For each prime $p\equiv 1 \pmod 5$, it appears that
\begin{equation}\label{5F4}
\pFq{5}{4}{\frac25,\frac25,\frac25,\frac25,\frac25}{,1,1,1,1}{1}_{p-1}\overset{?}{\equiv} -\G_p\left(\frac 15\right)^5\G_p\left(\frac 25\right)^5 \pmod {p^5}.
\end{equation}Using the strategy of the proof of Theorem 5 and Dougall's  formula \eqref{Dougall}, one can obtain \eqref{5F4} modulo $p^4$. By the Gross-Koblitz formula, the $p$-adic Gamma value  agrees with $J(\eta_n^2,\eta_n^2)^3J(\eta_n,\eta_n)$ when we choose the right order $n$ character. Meanwhile, by Conjecture \ref{thm:4}, we sense the presence of another Jacobi sum factor $-J(\eta_n,\eta_n)J(\eta_n,\eta_n^2)J(\eta_n,\eta_n^3)$.  It will be interesting to know whether we can reconstruct the local zeta function of $C_{5,1}$ as we have done for $C_{3,1}$ and $C_{4,1}$. We leave this task to interested readers.

We conclude with a few more observations.  Motivated by Lemma \ref{GaussToHyper}, we numerically  observed the following supercongruences. Each corresponds to a supercongruence mentioned earlier.

\begin{enumerate}
\item For any prime $p\equiv1 \pmod 3$,
\[
\sum_{k=0}^{p-1}  \(p\frac{k!}{(\frac53)_k} \)^3\equiv\sum_{k=\frac{2(p-1)}3}^{p-1}  \(p\frac{k!}{(\frac53)_k} \)^3\overset{?}{\equiv}
\G_p\left(\frac 13\right)^6 \pmod {p^3}.
\]

\item For any prime $p\equiv1 \pmod 4$,
\[
\sum_{k=0}^{p-1}  \(p\frac{k!}{(\frac74)_k} \)^4\equiv\sum_{k=\frac{3(p-1)}4}^{p-1}  \(p\frac{k!}{(\frac74)_k} \)^4\overset{?}{\equiv}
(-1)^{\frac{p-1}4}\G_p\left(\frac 12\right)\G_p\left(\frac 14\right)^6 \pmod {p^4}.
\]

\item For any prime $p\equiv1 \pmod 5$,

\[\sum_{k=0}^{p-1}  \(p\frac{k!}{(\frac85)_k} \)^5\equiv\sum_{k=\frac{3(p-1)}5}^{p-1}  \(p\frac{k!}{(\frac85)_k} \)^5\overset{?}{\equiv} -\G_p\left(\frac 15\right)^5\G_p\left(\frac 25\right)^5 \pmod {p^5}.\]

\item For an integer $n>2$, and any prime $p\equiv 1\pmod n$,

\[\sum_{k=0}^{p-1}  \(p\frac{k!}{(\frac1n+1)_k} \)^n\equiv\sum_{k=\frac{(p-1)}n}^{p-1}  \(p\frac{k!}{(\frac1n+1)_k} \)^n\overset{?}{\equiv}
-\G_p\left(\frac 1n\right)^n \pmod {p^3}.\]
\end{enumerate}

\section{Acknowledgements}
We warmly thank the Banff International Research Station (BIRS) and Women in Numbers 3 BIRS 2014 (14w5009)   workshop for the opportunity to initiate this collaboration.  Thanks to the National Center for Theoretical Sciences in Taiwan for supporting the travel of Fang-Ting Tu to visit Ling Long.  Long was supported by NSF DMS1303292.  We are indebted to Wadim Zudilin for his  insightful suggestions and sharing his ideas. We also thank Jerome W. Hoffman for his interest and valuable comments, Jes\'us Guillera and Ravi Ramakrishna for helpful discussions. Further thanks to the referees for their careful readings and helpful comments. We used \texttt{Magma} and \texttt{Sage} for our computations related to this project and \texttt{Sage Math Cloud} to collaborate.

\bibliographystyle{plain}
\bibliography{ref-bib}

\end{document}